\documentclass[11pt]{article}
\usepackage{color}
\usepackage{graphicx}
\usepackage{subfigure}
\usepackage{epsfig}
\usepackage{cite}
\usepackage{amsmath, amssymb}
\usepackage{fancybox}
\usepackage{listings}
\usepackage{verbatim}
\usepackage{url}
\usepackage{esint}
\usepackage{fancyhdr}
\usepackage{booktabs}
\usepackage{caption}
\usepackage{dsfont}
\input{epsf.sty}
\usepackage{mathrsfs,amsthm,amsmath}
\newtheorem{theo}{Theorem}[section]
\newtheorem{prop}{Proposition}[section]
\newtheorem{rem}{Remark}[section]

\newtheorem{cor}{Corollary}[section]

\newtheorem{ex}{Example}[section]

\begin{document}
\title{Piecewise linear processes with 
Poisson-modulated exponential switching times\footnote{This is the peer reviewed version of the paper, which has been published
in final form at https://doi.org/10.1002/mma.5683. This article can be used for non-commercial purposes in accordance with Wiley Terms and 
Conditions for use of Self-Archived versions.}}
\author{Antonio Di Crescenzo\thanks{Dipartimento di Matematica,
Universit\`a degli Studi di Salerno, 84084 Fisciano (SA), Italy. e-mail: \texttt{adicrescenzo@unisa.it}}
\and
Barbara Martinucci\thanks{Dipartimento di Matematica,
Universit\`{a} degli Studi di Salerno, 84084 Fisciano, SA, Italy. e-mail:
\texttt{bmartinucci@unisa.it}}\and Nikita
Ratanov\thanks{Facultad de Econom\'{\i}a,
Universidad del Rosario, Bogot\'a, Colombia. e-mail:
\texttt{nikita.ratanov@urosario.edu.co} (corresponding author)}
}
\date{}
\maketitle
\begin{abstract}{
We consider the jump telegraph process when
switching intensities depend on external shocks also accompanying
with jumps. The incomplete financial market model based on this process is studied. 
The Esscher transform, which changes only unobservable parameters, is considered in detail.
The financial market model based on this transform 
 can price switching risks as well as jump risks of the model. 
}

\end{abstract}

\section{Introduction and problem setting}
\setcounter{equation}{0}

The piecewise linear processes have a long history and still receiving 
attention in various aspects. This family of processes includes so-called 
telegraph process which presumes alternating velocities with exponentially 
 distributed   time intervals between switchings. The number of switchings 
 in this model is counted by homogeneous Poisson process.
 This theory has been developing 
since the seminal paper by Taylor \cite{Taylor} for almost a century. 
In the 50s, this model has been studied by Goldstein \cite{Gold}
and  Kac \cite{Kac}.  
See the history and the detailed description 
in  the monograph by Kolesnik  and  Ratanov \cite{KR}.

The model which is based on the distribution of the inter-switching times
different from exponential      is much less studied. Some examples could be found 
in \cite{DiC-Barbara-Zacks,DiCR,DiC-Zacks,damped,Weibull-logistic}.
Another generalisation can be constructed as a piecewise linear process 
with arbitrary consecutive 
trends $c_n$ and switching intensities $\lambda_n,\;n\geq0,\;\lambda_n>0,$  
see  \cite{STAPRO90} and \cite{ALEA17}. 

In this paper we study the piecewise linear processes
with exponentially distributed  time intervals between successive tendency switchings, but here we assume 
that the parameter of this distribution depends on  
exogenous shocks (exogenous impacts, external interventions), arriving at a constant rate.  
 This approach reflects the well posed problem of  financial market modelling, 
 when multiple agents are trying to break the trend by interventions, 
 but could affect only the  switching rate.
   
  Let $\lambda_n>0,\;n\geq0,$ be successive intensities of shocks.
  That is, on a complete probability space $(\Omega, \mathcal F, \mathbb{P})$
  we consider the sequence of time intervals $\tau_n,\;\tau_n>0,$ between the consecutive 
  shocks, which are independent and exponentially distributed,
  $\tau_n\sim\mathrm{Exp}(\lambda_n),\;n\geq0.$  
  
   We assume shocks do not come explosively, i. e. the
 process $\tau^{(+,n)}=\tau_0+\ldots+\tau_n$ is a \emph{simple point process}, 
 \[\mathbb P\{\lim_{n\to\infty}\tau^{(+,n)}=\infty\}=1,\] which is equivalent to
\begin{equation}
\label{eq:SPP}
\sum\lambda_n^{-1}=\infty,
\end{equation} 
 see e. g. 
\cite{jacobsen}. This means that there is a finite accumulation of shocks at any finite time interval:
 $\mathbb P\{N(t)<\infty\}=1,\;\forall t>0,$
 where $N=N(t),\;t\geq0,$ is the corresponding counting process,
  $N(t)=\sum_{n=1}^\infty\mathbf 1_{\{\tau^{(+, n-1)}\leq t\}},$ and $\mathbf 1_A$ is the indicator.
 
 If \eqref{eq:SPP} fails, then this process exibits explosive behaviour,
 $\mathbb P\{N(t)<\infty\}<1.$  For example, if $\lambda_n=(n+1)^2,$ then 
 $\mathbb P\{N(t)<\infty\}=-2\sum_{n=1}^\infty(-1)^n\mathrm e^{-n^2t}<1,$ see \cite[Example 6.3.1]{snyder}.

If all $\lambda'$s are equal, $\lambda_n\equiv\lambda,\;n\geq0,$  then $N=N(t)$ 
is the homogeneous Poisson process.  

  Let $\mu_n,\mu_n>0,\;n\geq0.$
  Consider the random variable $T,\;T>0,$ which has the exponential distribution 
  with $N(t)$-modulated parameter, $\mu=\mu_{N(t)}:$
   the survivor function of $T$ is given by
 \begin{equation}
\label{eq:barF}
\overline F_T(t):=\mathbb{P}\{T>t\}=\mathbb{E}\{\mathrm{e}^{-\xi(t)}\},
\end{equation}
where 
\begin{equation}
\label{def:xit}
\xi(t):=\int_0^t\mu_{N(s)}\mathrm{d} s
=\sum_{k=0}^{N(t)-1}\mu_k\tau_k
+\mu_{N(t)}\left(t-\sum_{k=0}^{N(t)-1}\tau_k\right)
=\sum_{k=0}^{N(t)-1}\left(\mu_k-\mu_{N(t)}\right)\tau_k
+\mu_{N(t)}t,\qquad t\geq0,
\end{equation}
is  the accumulated intensity. 
  
We call such a distribution a  \emph{Poisson-modulated exponential distribution}, 
${\rm PoExp}(\vec\lambda, \vec\mu).$ Here $\vec\lambda=\{\lambda_n\}_{n\geq0}$ is the sequence 
of parameters of the underlying process $N$ counting arrivals of shocks and 
$\vec\mu=\{\mu_n\}_{n\geq1}$ are the sequential parameters 
of the main exponential distribution. 

In this paper we study the piecewise linear process which follows two patterns
alternating at random instants
with  Poisson-modulated exponential  distribution  of 
inter-switching times.

To begin with, consider the sequence 
of independent Poisson processes $N=N_m(t),\;m\geq0,$
based on the two alternating sets of parameters 
$\vec\lambda^{(0)}=\{\lambda^{(0)}_n\}_{n\geq0}$ 
and $\vec\lambda^{(1)}=\{\lambda^{(1)}_n\}_{n\geq0},\; \lambda^{(0)}_n, \lambda^{(1)}_n>0,\;n\geq0,$
 \begin{equation}\label{def:Nm}
\mathbb{P}\{N_m(t)=n~|~N_m(s)=n\}=\exp\left(
-\lambda^{(\varepsilon_m)}_n(t-s)
\right),\qquad s<t,\quad n\geq0. 
\end{equation}
Here 
 $\varepsilon_m\in\{0, 1\}$ is a the sequence of alternating  $0$ and $1,$  which
 indicates the current pattern. 
Let these processes be non-explosive, 
satisfying \eqref{eq:SPP}, \[\sum_{n\geq0}{1/\lambda_n^{(i)}}=\infty,\qquad i\in\{0, 1\}.\]%

 Let  $\vec\mu^{(0)}=\{\mu_n^{(0)}\}_{n\geq0}$ and  $\vec\mu^{(1)}=\{\mu_n^{(1)}\}_{n\geq0}$ 
 be  two sequences of positive numbers.
 Consider the sequence 
 of independent random variables  $\{T_m\}_{m\geq1}$ , $T_m\geq0,$ 
 with alternating Poisson-modulated exponential  distributions, 
 $T_m\sim\mathrm{PoExp}(\vec\lambda^{(\varepsilon_m)}, \vec\mu^{(\varepsilon_m)}),\;m\geq1,$
 \[
 \mathbb P\{T_m>t\}=\mathbb E\exp\left(-\xi_m(t)\right),\qquad \xi_m(t)
 =\int_0^t\mu_{N_m(s)}^{(\varepsilon_m)}\mathrm ds.
 \]
 
 Consider the
point process formed by the times,  
$T^{(+,m)}=T_1+\ldots+T_m,\;{m\geq1},$  
when the patterns are switched, $T^{(+,0)}=0.$ 

 Since $N_m$ are non-explosive, $\mathbb P\{\forall t>0\; N_m(t)<\infty\}=1,\;m\geq1,$
the process $T^{(+,m)}$ is non-explosive too, if 
\[
\sum_{m\geq1}1/\mu_m^{(i)}=\infty,\qquad i\in\{0, 1\}.
\]
Let $M=M(t)\in\{0, 1,\ldots\}$ be
the counting  process   
\begin{equation}
\label{def:M}
M(t)=\max\{m~:~T^{(+,m)}\leq t\}, \qquad t>0.
\end{equation}

The marginal distributions of the process $\varepsilon=\varepsilon(t),\;t\geq0,$ 
indicating the current pattern, are defined by
\[
\mathbb P\{\varepsilon(t+\Delta t)=i~|~\varepsilon(t)=i,\;M(t)=m\}=\exp(-\mu_{N_m(t)}^{(i)}\cdot\Delta t)+o(\Delta t),
\qquad \Delta t\to0,\quad i\in\{0, 1\},
\]
$\varepsilon(t)=\varepsilon_m\in\{0, 1\},\;$ if $t\in[T^{(+,m)},\;T^{(+,m+1)}).$ 
By $\mathcal F_t,\;t\geq0,$ we denote the corresponding filtration.

Process $M=M(t)$ can be treated as a doubly stochastic Poisson process,
see \cite{Bremaud,Grandell}.
A similar approach is exploited in
\cite{Daley1}, see there Example 6.3(e), Example 7.3(a) and Example 7.4(e) 
(bivariate Poisson process).
Meanwhile this model differs from the model of mixed Poisson process 
(when the parameter $\mu$ of Poisson process 
is considered as the outcome of a positive $\mathcal F_0$-measurable 
random variable)
widely exploited in actuarial applications, see e.g. \cite{grandel,schmidt}.

The problem of infinite accumulation of arrivals for doubly stochastic Poisson process
(when the interarrival times are Poisson-modulated) is more complicated than for a simple point process,
see Section \ref{sec:PMED} for the analysis of this problem.

 The transport piecewise linear processes based on a mixed Poisson process 
have been recently presented in
\cite{STAPRO118,STAPRO129}. 
For the $2$D-case of such random motions
with jumps  (and with constant intensity of switchings)
see  \cite{RobertoEnzo}.
 The similar piecewise linear process $\mathbb L$ with
 the deterministically growing intensities $\mu^{(m)}=m\nu,\;\nu>0,$
has been studied in  \cite{DiC-Barbara-Zacks}. 

In this paper we study
 the piecewise linear renewal process $\mathbb L(t)$ based on the two alternating
sets of  tendencies, $\{c^{(0)}(n)\}_{n\geq0}$ and $\{c^{(1)}(n)\}_{n\geq0},$ 
with Poisson-modulated exponential distributions of patterns' holding times.

First, consider the sequence of independent piecewise linear processes 
 \[
 l_m(t)=\int_0^tc^{(\varepsilon_m)}(N_m(s))\mathrm{d} s,\qquad \varepsilon_m\in\{0, 1\},\quad m\geq0,
 \]
 which follows the  sequence of tendencies 
 ($\{c^{(0)}(n)\}_{n\geq0}$ or $\{c^{(1)}(n)\}_{n\geq0}$)
  with switchings at Poisson random times.
Process $\mathbb L(t)$ successively follows  the 
two patterns alternating after holding times $T_m:$
 \begin{equation}
\label{def:L}
\mathbb L(t) =\int_0^tl_{M(s)}(s)\mathrm{d} s
=\int_0^t\left(\int_0^sc^{(\varepsilon(s))}\left(N_{M(s)}(u)\right)\mathrm{d} u\right)\mathrm{d} s
=\sum_{m=1}^{M(t)}l_{m-1}(T_m)+l_{M(t)}\left(t-T^{(+, M(t))}\right), t\geq0.
\end{equation} 

Bearing in mind applications, we supply  $\mathbb L(t)$
 with  jumps.  Properties of process $\mathbb L$ with jumps and 
with exponentially distributed time intervals $T_m$ 
under the  arbitrary set of  tendencies $c_m$ 
is recently studied by \cite{STAPRO90}. The processes with 
renewal restarting points (instead of jumps) is studied by \cite{STAPRO131}.
Here we analyse 
properties of  such processes  
which follow the alternating patterns
with Poisson-modulated exponential  distributions of switching times 
and with jumps 
$r^{(i)}(n)$ and $R^{(i)}(n),\;n\geq0,$ $  i\in\{0, 1\},$
 accompanying the tendency switchings and the changes of patterns, respectively.

This approach could be used for financial market modelling
when log-returns are 
determined by the inherent market forces, that is by $\mathbb L(t),\;\{R(n)\},$ and 
by efforts of ``small speculators'' which create the tendency and volatility modulations. 
Let jumps $R(n)$ accompanying the patterns' switchings  occur
as ``corrections'' of the current trend $c(n)$:
\[
c(n)/\mathbb{E}\{R(n)\}<0.
\]
In this case small players trying to change the trend affect only the volatility
and the probability of the next switching of the trend.

We derive coupled integral equations 
 for mean values of the process $\mathbb L$ accompanied with jumps
(see Section \ref{sec:Cox/Hawkes  processes}). Further, the martingale condition is presented: 
this process is a martingale if and only if 
\begin{equation}\label{eq:mart}
c^{(i)}(n)+\lambda_n^{(i)}\overline{r^{(i)}(n)}+\mu_n^{(i)}\overline{R^{(i)}(n)}\equiv0,
\qquad n\geq0,\quad i\in\{0, 1\}.
\end{equation}
Here $\overline{r^{(i)}(n)}$ and $\overline{R^{(i)}(n)}$ are the expectations of the random jump values.
The same condition characterises martingales in a rather different model, 
when holding times $T_m$
are independent of the underlying Poisson processes $N_m,$ 
 see  \cite[Corollary 3.1]{ALEA17}.
Condition \eqref{eq:mart} looks similar to the martingale condition for a simple jump-telegraph process, see  \cite{Quant}:
\begin{equation}\label{eq:mart-classic}
c^{(i)}+\lambda^{(i)} r^{(i)}=0,\quad i\in\{0, 1\}.
\end{equation}

The text is organised as follows.
The detailed analysis of Poisson-modulated exponential  distributions
is presented in Sections \ref{sec:PMED} and   \ref{Sec:damped}. 
In Section \ref{sec:Cox/Hawkes  processes} we study the
piecewise linear process $\mathbb L$ with jumps.

In Section \ref{sec:market} we propose a financial market model based on these processes. 
This model generalises the simple jump-telegraph market model, \cite{Quant}.

\section{Poisson-modulated exponential  distribution}\label{sec:PMED}
\setcounter{equation}{0}

In this section we present some properties of the Poisson modulated exponential distributions
which we will use later.

Let  $\{\tau_n\}_{n\geq0}$ be   the  sequence
of independent and 
exponentially distributed  
$\mathrm{Exp}(\lambda_n),$ \break $\lambda_n>0,\; n\geq0,$ random variables, 
and $N=N(t)=N(t; \vec\lambda),\;t\geq0,$ be 
the  non-explosive, \eqref{eq:SPP}, renewal  counting process.
 Denote the probability mass function of $N(t)$
at $n$ by $\pi_n(t)=\mathbb{P}\{N(t)=n\},\;n\geq0$. If all $\lambda_n$ are equal, $\lambda_n\equiv\lambda,$ then 
$\pi_n(t)=\mathrm{e}^{-\lambda t}(\lambda t)^n/n!;$ 
in the case of distinct $\lambda_n$ 
 by \cite[formula (2.4), Proposition 2.1]{STAPRO90}, we have 
\begin{equation}
\label{eq:pin}
\pi_n(t)=\Lambda_na_n(t; \vec\lambda),
\end{equation}
where $\Lambda_n=\prod_{k=0}^{n-1}\lambda_k,\;n\geq1,\;\Lambda_0=1,$
\begin{equation}\label{def:a}
a_n(t)=a_n(t; \vec\lambda)
=\sum_{k=0}^{n}\kappa_{n,k}(\vec\lambda)\mathrm{e}^{-\lambda_kt},\qquad t>0,
\end{equation}
and coefficients $\kappa_{n,k}(\vec\lambda)$
are defined by
\begin{equation}
\label{def:kappa}
\kappa_{n,k}(\vec\lambda)
=\prod\limits_{\stackrel{j=0}{j\neq k}}^{n}(\lambda_j-\lambda_k)^{-1},\;n\geq k,
\qquad \kappa_{0,0}=1.
\end{equation}
In the non-explosive case, \eqref{eq:SPP},
\[
\sum_{n}\pi_n(t)\equiv1,\qquad\forall t>0.
\]
In the case when not all $\lambda'$s are distinct the usual changes in notations
 should be applied, see \cite{STAPRO90}.

Notice, that coefficients $\kappa_{n,k}$ satisfy the following known 
Vandermonde properties\textup{:}  for  $n\geq1$
\begin{equation}
\label{Vandermonde}
   \sum_{k=0}^n\kappa_{n, k}(\vec\lambda) \lambda_k^m=0, \quad 0\leq m\leq n-1, \qquad
     \sum_{k=0}^n\kappa_{n, k}(\vec\lambda) \lambda_k^n =(-1)^n,
\end{equation}
(see e.g. \cite{Kuznetsov}, p. 11).

Due to identities \eqref{Vandermonde}, functions 
$a_n=a_n(t)=\sum_{k=0}^{n}\kappa_{n,k}(\vec\lambda)\exp(-\lambda_kt),$ 
satisfy the following conditions: 
 for $n\geq1$
 \begin{equation}\label{eq:a}
 a_n(0)=0,\qquad
\frac{\mathrm{d}^m a_n(t)}{\mathrm{d} t^m}|_{t=0}=0,\quad 0<m\leq n-1,
\qquad \text{and}\qquad\frac{\mathrm{d}^n a_n(t)}{\mathrm{d} t^n}|_{t=0}=1;
\end{equation}
$a_0(t)=\exp(-\lambda_0t).$

Consider the random variable $T,\;T>0,$ 
with Poisson-modulated exponential distribution, ${\rm PoExp}(\vec\lambda, \vec\mu).$ 
See \eqref{eq:barF}-\eqref{def:xit}. 

Let us begin by studying the properties of the accumulated intensity $\xi(t),$ 
based on the non-explosive counting process $N(t)$, 
\eqref{def:xit}.
Consider the moment generating function  $\psi(z,t)$ of  $\xi(t),$ 
\begin{equation}
\label{eq:MGFn}
\psi(z, t)=\mathbb{E}\{\mathrm{e}^{-z\xi(t)}\}
=\sum_{n=0}^\infty\psi_n(z, t),
\end{equation}
where 
\begin{equation}\label{def:psin1}
\psi_n(z, t)=\mathbb{E}\{\mathrm{e}^{-z\xi(t)}\mathbf{1}_{N(t)=n}\},
\qquad t>0,
\qquad n\geq0.
\end{equation}
By definition we have 
\begin{equation}
\label{eq:T>t-n}
\mathbb{P}\{T>t,\; N(t)=n\}=\psi_n(1, t).
\end{equation}

Assume that all linear $z$-functions
\begin{equation}\label{def:tilde}
z\to\lambda_n+z\mu_n=\tilde\lambda_n(z)=\tilde\lambda_n,\;n\geq0,
\end{equation}
 are distinct,
$\tilde\lambda_n
\neq
\tilde\lambda_k$  
(if  $n\neq k$)
one can easily obtain the explicit expression  for $\psi_n$
by means of $a_k,\;k\leq n.$ Applying  \cite[Theorem 3.1]{STAPRO90},
we have 
\begin{equation}
\label{eq:psin1}
\psi_n(z, t)=\Lambda_n\sum_{k=0}^n\kappa_{n,k}(\tilde\lambda)\mathrm{e}^{-\tilde\lambda_kt}
=\Lambda_na_n(t;\tilde\lambda),
\end{equation}
where $\kappa_{n,k}(\tilde\lambda)$ and $a_n(t;\tilde\lambda)$
are defined by \eqref{def:kappa} and \eqref{def:a} 
with $\tilde\lambda=\tilde\lambda(z)$ instead of $\vec\lambda$.

We introduce the notation 
\begin{equation}
\label{def:bk}
b_k(\vec\lambda, z\vec\mu):=
\sum_{n=k}^\infty\Lambda_n\kappa_{n, k}(\tilde\lambda)<\infty,\qquad k\geq0,
\end{equation}
assuming convergence of the series. 
From \eqref{eq:psin1}  
one can obtain the representation of $\psi(z, t),$ which is equivalent to \eqref{eq:MGFn}: 
  \[
  \psi(z, t)=\sum_{n=0}^\infty\psi_n(z, t)
  =\sum_{n=0}^\infty\Lambda_n\sum_{k=0}^n
  \kappa_{n,k}(\tilde\lambda)\mathrm{e}^{-\tilde\lambda_kt}
  =\sum_{k=0}^\infty\left(\sum_{n=k}^\infty\Lambda_n\kappa_{n, k}(\tilde\lambda)\right)
  \mathrm{e}^{-\tilde\lambda_k t}
  =\sum_{k=0}^\infty b_k(\vec\lambda, z\vec\mu)\mathrm{e}^{-\tilde\lambda_kt},
  \]
if the series in \eqref{def:bk} converge. 
This representation is consistent with the identities 
$\psi(z, 0)=\mathbb E\{\mathrm e^{-z\xi(t)}|_{t=0}\}\equiv1$ 
and $\psi(0, t)=\mathbb E\{\mathrm e^{-z\xi(t)}|_{z=0}\}\equiv1$:
\[
\psi(z, 0)=\sum_{k=0}^\infty b_k(\vec\lambda, z\vec\mu)
=\sum_{n=0}^\infty\Lambda_n\sum_{k=0}^n\kappa_{n, k}(\tilde\lambda)
=1+\sum_{n=1}^\infty\Lambda_n\sum_{k=0}^n\kappa_{n, k}(\tilde\lambda)\equiv1,
\]
since due to \eqref{Vandermonde} (with $m=0$)  
$\sum_{k=0}^n\kappa_{n,k}(\tilde\lambda)\equiv0,\;n\geq1;$
 and by \eqref{eq:pin}-\eqref{def:a}
 \[
 \psi(0, t)=\sum_{k=0}^\infty \mathrm e^{-\lambda_kt}\sum_{n=k}^\infty\Lambda_n\kappa_{n, k}(\vec\lambda)
 =\sum_{n=0}^\infty\Lambda_n\sum_{k=0}^n\kappa_{n, k}(\vec\lambda)\mathrm e^{-\lambda_kt}
 =\sum_{n=0}^\infty\pi_n(t)\equiv1.
 \]

\begin{rem}
 Notice that equalities \eqref{eq:MGFn}-\eqref{def:psin1} and \eqref{eq:psin1} 
return us to some known formulae for the  
$\lambda_n,\;\lambda_n\equiv\lambda,\;\lambda>0.$
 
By \eqref{eq:psin1}  and \eqref{def:kappa}  we have
\begin{equation*}
\psi_n(z, t)=\lambda^nz^{-n}\sum_{k=0}^n\kappa_{n, k}(\vec\mu)\mathrm{e}^{-z\mu_kt} \mathrm{e}^{-\lambda t}
=\lambda^nz^{-n}a_n(zt) \mathrm{e}^{-\lambda t},
\end{equation*}
which due to \eqref{eq:a}
gives 
\[
\pi_n(t)=\psi_n(z, t)|_{z\downarrow0}=\lambda^n\mathrm e^{-\lambda t}
\cdot\lim_{z\downarrow0}\left[z^{-n}a_n(zt)\right]
=\lambda^n\mathrm e^{-\lambda t}
\lim_{z\downarrow0}\left[z^{-n}\frac{\mathrm{d}^n a_n(t)}{\mathrm{d} t^n}|_{t\downarrow0}\frac{(zt)^n}{n!}\right]
=\frac{(\lambda t)^n}{n!}\mathrm{e}^{-\lambda t}.
\]

Further, by \eqref{def:tilde}    one has
\[
b_k(\lambda, z\vec\mu)=\sum_{n=k}^\infty z^{-n}\lambda^n\kappa_{n,k}(\vec\mu).
\]

Note that the series convergence here and in \eqref{def:bk} occurs, for example, if
\[
|\mu_n-\mu_m|\geq\nu>0,\qquad n\neq m.
\]
Indeed, $|\kappa_{n, k}(\vec\mu)|\leq\nu^{-n}$ and 
\[
|b_k(\lambda, z\vec\mu)|\leq\sum_{n=k}^\infty\left(\frac{\lambda}{|z|\nu}\right)^n
=\frac{\left(\lambda/|z|\nu\right)^k}{1-\lambda/|z|\nu}<\infty
\]
for $|z|>\lambda/\nu$.

In the linear case,  $\mu_n=\mu+n\nu$ \textup{(}see Section \textup{\ref{Sec:damped}),}
\[
\kappa_{n,k}(\tilde\lambda) 
=z^{-n}\frac{(-1)^k\nu^{-n}}{k!(n-k)!}
\]
and 
\[
b_k(\lambda, z\vec\mu)
=\sum_{n=k}^\infty\frac{(-1)^k\lambda^n(z\nu)^{-n}}{k!(n-k)!}
=\frac{(-\lambda/z\nu)^k}{k!}\mathrm{e}^{\lambda/z\nu}<\infty,\qquad z\neq0.
\]
\end{rem}

We express the distribution of $T$ in these terms,
starting with the following useful formulas.

\begin{prop}\label{prop-densities}
If all sums $\lambda_n+\mu_n,\;n\geq0,$ are distinct, then
\begin{align}
\label{1}
   \mathbb{P}\{T>t,\;N(t)=n\} &=  \Lambda_n a_n(t; \overrightarrow{\lambda+\mu}),    \\
    \mathbb{P}\{T\in\mathrm{d} t,\;N(t)=n\}&=  \mu_n\Lambda_na_n(t;  \overrightarrow{\lambda+\mu})  \mathrm{d} t,
    \label{2}
\end{align}
$t\geq0,$ where functions $a_n(t; \overrightarrow{\lambda+\mu})$ are defined by \eqref{def:a}.
\end{prop}
\begin{proof}
Equality \eqref{1} follows from \eqref{eq:T>t-n} and \eqref{eq:psin1}.
   
   To prove \eqref{2} note that by \cite[$\left(\right.$formula (2.5)$\left.\right)$]{STAPRO90} the 
   joint distribution of $(\tau_0,\ldots,\tau_{n-1})\mathbf{1}_{\{N(t)=n\}}$     is given by
   \begin{equation*}
\mathbb{P}\{\tau_0\in\mathrm{d} s_0,\ldots,\tau_{n-1}\in\mathrm{d} s_{n-1},\; N(t)=n\}=
\Lambda_n\mathrm{e}^{-\lambda_nt}\exp\left(
-\sum_{k=0}^{n-1}(\lambda_k-\lambda_n)s_k
\right)\mathbf{1}_{\Xi_n(t)}(\vec s)\mathrm{d}\vec s,
\end{equation*}
where  
$\mathbf{1}_{\Xi_n(t)}(\vec s)=\begin{cases}
     1, & \text{ if } \vec s\in\Xi_n(t) \\
     0, & \text{otherwise}
\end{cases},$
where
$\Xi_n(t):=\{\vec s=(s_0,\ldots, s_{n-1})\in\mathbb{R}^n_+~|~s^{(+,n)}=s_0+\ldots+s_{n-1}<t\}.$

By \eqref{eq:barF}-\eqref{def:xit}
   \begin{equation*}
\begin{aligned}
&\mathbb{P}\{T\in\mathrm{d} t,\; N(t)=n\}/\mathrm{d} t
=\mathbb E\left\{\mu_{N(t)}\exp\left(-\xi(t)\right)\mathbf{1}_{\{N(t)=n\}}\right\}\\
=&\mu_n\mathrm e^{-\mu_nt}\int_{\Xi_n(t)}\exp\left(
-\sum_{k=0}^{n-1}(\mu_k-\mu_n)s_k
\right)\cdot
\Lambda_n\mathrm{e}^{-\lambda_nt}\exp\left(
-\sum_{k=0}^{n-1}(\lambda_k-\lambda_n)s_k
\right)
\mathrm{d}\vec s.
\end{aligned}\end{equation*}

Applying  \cite[formula (3.5)]{STAPRO90} we get \eqref{2}. 
\end{proof}

If some of $\lambda_n+\mu_n$ are equal, the usual changes should be applied
(for details, see \cite{ALEA17}).

\begin{cor}
Let the series \eqref{def:bk} converges.

Under the conditions of Proposition \ref{prop-densities} 
\begin{itemize}
\item
the survivor function $\overline F_T(t)=\mathbb{P}\{T>t\}$ has the form 
\begin{equation}
\label{eq:FT}
\overline F_T(t)=\sum_{k=0}^\infty b_k(\vec\lambda, \vec\mu)\mathrm{e}^{-(\lambda_k+\mu_k)t},\quad t\geq0;
\end{equation}
\item
the density function $f_T(t)$ is given by
\begin{equation}\label{eq:fT1}
f_T(t)=\sum_{k=0}^\infty (\lambda_k+\mu_k)b_k(\vec\lambda, \vec\mu)\mathrm{e}^{-(\lambda_k+ \mu_k)t},
\end{equation}
if the series in \eqref{eq:fT1} converges.
In particular,
$f_T(0)=\mu_0$.
\end{itemize}
\end{cor}

\begin{proof}
If    \eqref{def:bk} holds, under the conditions of Proposition \ref{prop-densities}
the survivor function   
takes the form
\begin{equation*}
\overline F_T(t)
=\sum_{n=0}^\infty\Lambda_na_n(t; \overrightarrow{\lambda+\mu})
=\sum_{n=0}^\infty\Lambda_n\sum_{k=0}^n
\kappa_{n, k}(\overrightarrow{\lambda+\mu})\mathrm{e}^{-(\lambda_k+\mu_k)t},\quad t\geq0,
\end{equation*}
which gives \eqref{eq:FT}.
By differentiation the density function takes the form \eqref{eq:fT1},
 if the series in \eqref{eq:fT1} converges.

On the other hand, by \eqref{2} we have 
\begin{equation}
\label{eq:fT2}
f_T(t)=\sum_{n=0}^\infty\mu_n\Lambda_na_n(t; \overrightarrow{\lambda+\mu})
=\sum_{k=0}^\infty\mathrm{e}^{-(\lambda_k+ \mu_k)t}
\sum_{n=k}^\infty\mu_n\Lambda_n\kappa_{n, k}(\overrightarrow{\lambda+\mu}).
\end{equation}
Representations \eqref{eq:fT1} and \eqref{eq:fT2} are equivalent.
Indeed, by definition  \eqref{def:kappa} of $\kappa_{n, k}:$ 
\[\begin{aligned}
\sum_{n=k}^\infty&\mu_n\Lambda_n\kappa_{n, k}(\overrightarrow{\lambda+\mu})
=\sum_{n=k}^\infty(\mu_n+\Lambda_n)\lambda_n\kappa_{n, k}(\overrightarrow{\lambda+\mu})
-\sum_{n=k}^\infty\Lambda_{n+1}\kappa_{n,k}(\overrightarrow{\lambda+\mu})\\
&=\sum_{n=k+1}^\infty\Lambda_n\kappa_{n-1, k}(\overrightarrow{\lambda+\mu})
+(\mu_k+\lambda_k)\sum_{n=k}^\infty\Lambda_n\kappa_{n, k}(\overrightarrow{\lambda+\mu})
-\sum_{n=k}^\infty\Lambda_{n+1}\kappa_{n,k}(\overrightarrow{\lambda+\mu})\\
&=(\mu_k+\lambda_k)\sum_{n=k}^\infty\Lambda_n\kappa_{n, k}(\overrightarrow{\lambda+\mu}),
\end{aligned}\]
since \eqref{def:bk} holds.

Moreover, in this case by \eqref{def:bk}
one can see that
\begin{equation*}
\begin{aligned}
\sum_{k=0}^\infty (\lambda_k+\mu_k) b_k(\vec\lambda, \vec\mu)
=&\sum_{n=0}^\infty\Lambda_n\sum_{k=0}^n(\lambda_k+\mu_k)\kappa_{n, k}(\tilde\lambda)\\
=&(\lambda_0+\mu_0)+\lambda_0\sum_{k=0}^1(\lambda_k+\mu_k)\kappa_{1, k}(\tilde\lambda)
+\sum_{n=2}^\infty\Lambda_n\sum_{k=0}^n(\lambda_k+\mu_k)\kappa_{n, k}(\tilde\lambda).\end{aligned}
\end{equation*}
Therefore, by \eqref{Vandermonde}
we have 
\begin{equation*}
\sum_{k=0}^\infty (\lambda_k+\mu_k) b_k(\vec\lambda, \vec\mu)
=\mu_0.
\end{equation*}
Hence $f_T(0)=\mu_0$.
\end{proof}

The convergence in \eqref{def:bk} plays the role of a non-exploding condition.
\begin{prop}
Let $T$ be a Poisson-modulated exponential, $\mathrm{PoExp}(\vec\lambda, \vec\mu),$ random variable.

The distribution of variable $T$ is proper, that is
\[
\mathbb P\{T<\infty\}=1,
\]
if \eqref{def:bk} holds.
\end{prop}
\begin{proof}
Let \eqref{def:bk} holds.
By integrating in \eqref{eq:fT1},
\[
\int_0^\infty f_T(t)\mathrm dt=\sum_{k=0}^\infty b_k(\vec\lambda, \vec\mu)
=\sum_{n=0}^\infty\Lambda_n\sum_{k=0}^n\kappa_{n, k}(\overrightarrow{\lambda+\mu})
=1+\sum_{n=1}^\infty\Lambda_n\sum_{k=0}^n\kappa_{n, k}(\overrightarrow{\lambda+\mu})
=1.
\]
\end{proof}

The moments of random variable $T$ can be obtained similarly. 
\begin{prop}
If the distribution of $T$ satisfies 
\begin{equation}\label{eq:ETcond}
\lim_{t\to\infty}t \overline F_T(t)=0,
\end{equation}
and the series $\sum_{k=0}^\infty b_k(\vec\lambda, \vec\mu)(\lambda_k+\mu_k)^{-1}$ converges, then
the expectation $\mathbb E\{T\}$ exists and
\[
\mathbb{E}\{T\}=\sum_{k=0}^\infty b_k(\vec\lambda, \vec\mu)(\lambda_k+\mu_k)^{-1}
=\sum_{n=0}^\infty\Lambda_n\Pi_n^{-1}
<\infty,
\]
where $\Pi_n=\prod_{k=0}^n(\lambda_k+\mu_k).$

Moreover,
 if for some $m,\; m\geq1,$
\begin{equation*}
\lim_{t\to\infty}t^m \overline F_T(t)=0,
\end{equation*} 
and the series $\sum_{k=0}^\infty b_k(\vec\lambda, \vec\mu)(\lambda_k+\mu_k)^{-m}$ converges,
then 
\begin{equation}\label{eq:EETm}
\mathbb{E}\{T^m\}=
m!\sum_{k=0}^\infty b_k(\vec\lambda, \vec\mu)(\lambda_k+\mu_k)^{-m}<\infty.
\end{equation}
\end{prop}

\begin{proof}
Due to \eqref{eq:FT} and \eqref{eq:ETcond}
\[
\mathbb{E}\{T\}=\int_0^\infty  \overline F_T(t)\mathrm{d} t
=\sum_{k=0}^\infty b_k(\vec\lambda, \vec\mu)(\lambda_k+\mu_k)^{-1}
=\sum_{n=0}^\infty\Lambda_n\sum_{k=0}^n\kappa_{n, k}(\overrightarrow{\lambda+\mu})(\lambda_k+\mu_k)^{-1}.
\]
Notice that  by the Vandermonde properties, see \cite{Kuznetsov},     
\[
\sum_{k=0}^n \kappa_{n,k}(\overrightarrow{\lambda+\mu})(\lambda_k+\mu_k)^{-1}
=\Pi_n^{-1}.
\] 
Therefore 
\[
\mathbb{E}\{T\}=\sum_{n=0}^\infty\Lambda_n\Pi_n^{-1}
=\sum_{n=0}^\infty\left[\prod_{k=0}^{n-1}\frac{\lambda_k}{\lambda_k+\mu_k}\right]
\times\frac{1}{\lambda_n+\mu_n}.
\]

In general, if for some $m,\; m\geq1,$
\begin{equation*}
\lim_{t\to\infty}t^m \overline F_T(t)=0,
\end{equation*} 
then  integrating by parts we obtain
\begin{equation*}\begin{aligned}
\mathbb{E}\{T^m\}=&-\int_0^\infty t^m\mathrm{d} \overline F_T(t)
=m\int_0^\infty t^{m-1}\sum_{n=0}^\infty\Lambda_n
\sum_{k=0}^n\kappa_{n, k}(\overrightarrow{\lambda+\mu})\mathrm{e}^{-(\lambda_k+\mu_k)t}\mathrm{d} t\\
=&m\sum_{n=0}^\infty\Lambda_n\sum_{k=0}^n\kappa_{n,k}(\overrightarrow{\lambda+\mu})
\int_0^\infty t^{m-1}\mathrm{e}^{-(\lambda_k+\mu_k)t}\mathrm{d} t
=m!\sum_{n=0}^\infty\Lambda_n\sum_{k=0}^n\kappa_{n,k}(\overrightarrow{\lambda+\mu})
(\lambda_k+\mu_k)^{-m}\\
=&m!\sum_{k=0}^\infty b_k(\vec\lambda, \vec\mu)(\lambda_k+\mu_k)^{-m}<\infty.
\end{aligned}\end{equation*}
The latter equality follows by \eqref{def:bk}.
\end{proof}

We conclude the section with some examples.

\begin{ex}
\begin{description}
  \item[(a)] Let $\lambda_n=\mu_n=\frac12(n+1)^2,$
  that is $\sum_n\lambda_n^{-1}=\sum_{n}\mu_n^{-1}<\infty.$ By \eqref{def:kappa} and \eqref{def:bk} we have 
  \[
  \kappa_{n, k}(\overrightarrow{\lambda+\mu})=\left(\prod\limits_{j=0,\;j\neq k}^n[(j+1)^2-(k+1)^2]\right)^{-1}
  =(-1)^k\frac{2(k+1)\cdot(k+1)!}{k!(n-k)!(n+k+2)!},
  \]
  \[
  \Lambda_n=(n!)^22^{-n},\quad\Pi_n=[(n+1)!]^2,\quad
  b_k(\vec\lambda, \vec\mu)
  =(-1)^k2(k+1)^2\cdot\sum_{n=k}^\infty\frac{(n!)^2}{(n-k)!(n+k+2)!}2^{-n}<\infty
  \]
    and
    \[
    \mathbb E\{T\}=\sum_{n=0}^\infty\Lambda_n\Pi_n^{-1}=\sum_{n=0}^\infty\frac{2^{-n}}{(n+1)^2}<\infty.
    \]
  \item[(b)] Let $\lambda_n=n+1,\;\mu_n=1,$ 
  that is $\sum_{n}\lambda_n^{-1}=\infty,\;\sum_{n}\mu_n^{-1}=\infty,$
  \[
  \kappa_{n, k}=\frac{(-1)^k}{k!(n-k)!},\qquad \Lambda_n=n!,\quad\Pi_n=(n+2)!,\quad
  b_k(\vec\lambda, \vec\mu)=\sum_{n=k}^\infty n!\frac{(-1)^k}{k!(n-k)!}=\infty
  \]
  and
  \[
  \mathbb E\{T\}=\sum_{n=0}^\infty\Lambda_n\Pi_n^{-1}=\sum_{n=0}^\infty\frac{1}{(n+1)(n+2)}=1.
  \]
  \item[(c)]  Let $\lambda_n=1,\;\mu_n=\dfrac{1}{n+1},$ 
  that is $\sum_{n}\lambda_n^{-1}=\infty,\;\sum_{n}\mu_n^{-1}=\infty,$
  \[
  \kappa_{n, k}=\frac{(-1)^{n-k}(k+1)^{n-1}n!}{k!(n-k)!},\quad
  \Lambda_n=1,\quad\Pi_n=n+2,\quad b_k(\vec\lambda, \vec\mu)=\infty,\quad \mathbb E\{T\}=\infty.
  \]
\end{description}
\end{ex}

An example with a linearly increasing intensity $\mu_n$ is given in the next section.
\section{Example: 
 Poisson-modulated exponential  distributions with linearly increasing switching intensities}
 \label{Sec:damped}
\setcounter{equation}{0}

All the formulae can be simplified  and detailed 
in the case of a linear  increase in switching intensities, $\mu_n=\mu+n\nu$. 

\begin{theo}\label{theo:T}
Let  random variable $T$ has a Poisson-modulated exponential distribution
with a homogeneous underlying Poisson process, 
$\lambda_n\equiv\lambda,\; T\sim\mathrm{PoExp}(\lambda,\vec\mu),$
where $\vec\mu=\{\mu+n\nu\}_{n\geq0}$ with some $\nu,\;\nu>0.$

The cumulative distribution function of $T$ is given by 
\begin{equation}
\label{eq:FTHawkes}
F_T(t)=\left(
1-\mathrm{e}^{-\mu t-\lambda \mathcal A(t)}
\right)\mathbf{1}_{t\geq0}
\end{equation}
and  the moment generating function is
\begin{equation}
\label{eq:MGF}
\Psi(z):=\mathbb{E}[\mathrm{e}^{-zT}]
=1-z\mathrm{e}^{\lambda/\nu}\sum_{n=0}^\infty\frac{(-1)^n(\lambda/\nu)^n}{n!(\lambda+\mu+z+n\nu)}.
\end{equation}
\end{theo}

Here we denote 
$\mathcal A(t)=\int_0^t\alpha(u)\mathrm{d} u=t-\alpha(t)/\nu,\;t\geq0,$
where $\alpha(t)=1-\mathrm{e}^{-\nu t}$.

\begin{proof}
First, note that in this case
\begin{equation}\label{eq:kappa-ex}
\kappa_{n,k}=\left[\prod\limits_{j=0\;j\neq k}^n(j-k)\nu\right]^{-1}
=\frac{(-1)^k}{k!(n-k)!}\nu^{-n}
\end{equation}
and
\begin{equation}\label{eq:bk-ex}
b_k
=\sum_{n=k}^\infty\left(\lambda/\nu\right)^n\frac{(-1)^k}{k!}\frac{1}{(n-k)!}
=\frac{(-1)^k}{k!}\left(\lambda/\nu\right)^k\mathrm e^{\lambda/\nu}.
\end{equation}
By \eqref{eq:FT} and \eqref{eq:bk-ex}
\[
\overline F_T(t)=1-F_T(t)=\sum_{k=0}^\infty b_k\mathrm e^{-(\lambda_k+\mu_k)t}
=\mathrm e^{\lambda/\nu}\sum_{k=0}^\infty\frac{(-1)^k}{k!}
\left(\lambda/\nu\right)^k\mathrm e^{-(\lambda+\mu+k\nu)t}
=\exp\left(\frac{\lambda}{\nu}\left(1-\mathrm e^{-\nu t}\right)-(\lambda+\mu)t\right),
\]
which coincides with \eqref{eq:FTHawkes}.

Therefore,
\begin{equation}\label{eq:PsiI}
\begin{aligned}
\Psi(z)=&\int_0^\infty\mathrm{e}^{-zt}\mathrm{d} F_T(t)
=\int_0^\infty\left(\mu+\lambda\alpha(t)\right)\exp\left(-(\mu+z)t-\lambda\mathcal A(t)\right)\mathrm dt\\
=&(\mu+\lambda)I(z)-\lambda I(z+\nu),
\end{aligned}
\end{equation}
where 
\[
I(z)=\int_0^\infty\exp\left(-(\mu+z)t-\lambda\mathcal A(t)\right)\mathrm dt
=\frac{1}{\nu}\int_0^1x^{b-1}\mathrm e^{-\beta x}\mathrm dx,\qquad 
b=b(z)=\frac{\mu+z+\lambda}{\nu},\quad \beta =\frac{\lambda}{\nu}.
\]
Due to  \cite{GR}, formula (3.381.1), 
\begin{equation}
\label{eq:I}
I(z)=\frac{\gamma(b(z), \beta)}{\nu\beta^b}.
\end{equation}

Since $b(z+\nu)=b(z)+1$,  
by using \cite{GR}, formula (8.356.1), from \eqref{eq:PsiI} we get
\[
\Psi(z)=1-\frac{z\beta^{-b(z)}\mathrm{e}^\beta}{\nu}\gamma(b(z), \beta).
\]
Formula \eqref{eq:MGF}
follows by the series representation
of the incomplete gamma-function, see \cite{GR}, formula (8.354.1).
\end{proof}

\begin{rem}
The distribution of $T,$ given by \eqref{eq:FTHawkes} is unimodal, that is, 
the corresponding density function 
\begin{equation}
\label{eq:density}
f_T(t)=\left(
\mu+\lambda\alpha(t)
\right)\exp(-(\mu+\lambda) t+\lambda\alpha(t)/\nu),\qquad t>0,
\end{equation}
has a single maximum at  point $m,$
where $m=0$ if $\lambda\nu\leq \mu^2,$ and 
\[
m=\frac{1}{\nu}\ln\left[
\frac{\lambda}{\lambda+\mu}\left\{1+\frac{\nu}{2(\lambda+\mu)}
+\sqrt{\frac{\nu}{\lambda+\mu}\left(1+\frac{\nu}{4(\lambda+\mu)}\right)}\right\}
\right]
\]
if $\lambda\nu>\mu^2$.
See Figure \ref{fig3}.

\begin{figure}[ht]
\begin{center}
  \includegraphics[scale=0.8]{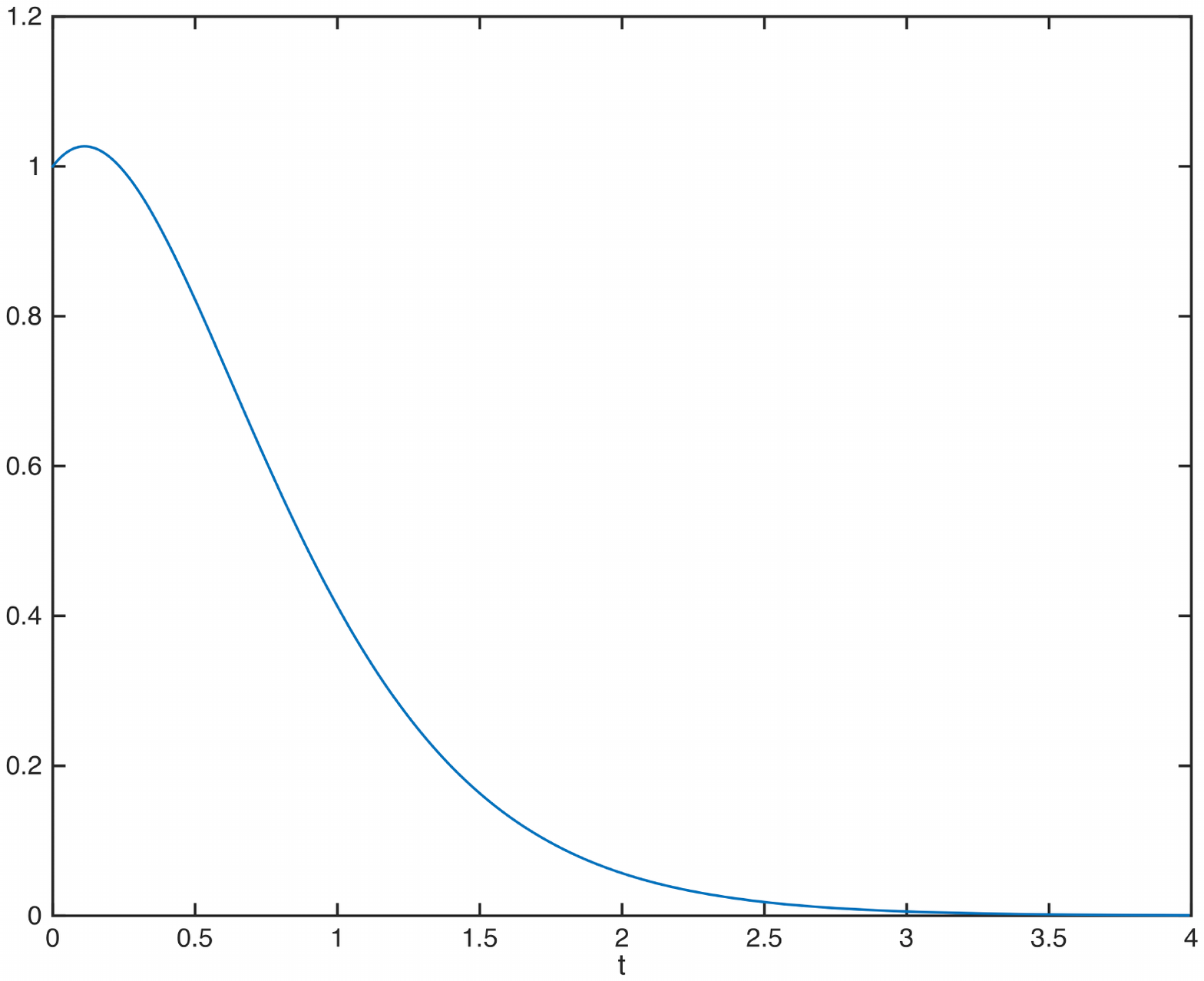}\\
 \caption{Density function the Poisson modulated distribution, (3.7), 
 with $\lambda=1.5,\; \mu=1,\;\nu=1$
  }\label{fig3}
 \end{center}
\end{figure} 
\end{rem}

The moments could be computed by differentiating the 
moment generating function $\Psi(z),$ which is given
 by  \eqref{eq:MGF},
\begin{equation}\label{eq:ETm-ex}
\mathbb{E}\{T^m\}=(-1)^m\Psi^{(m)}(0)
=\mathrm{e}^{\lambda/\nu}m!
\sum_{n=0}^\infty\frac{(-1)^n(\lambda/\nu)^n}{n!(\lambda+\mu+n\nu)^{m}},
\qquad m=1, 2,\ldots
\end{equation}
The same result follows from \eqref{eq:EETm} and \eqref{eq:bk-ex}. 

Formula \eqref{eq:ETm-ex} can be simplified: by 
\cite[(9.14)]{GR} 
using a 
generalised hypergeometric function \newline
$_mF_m(\alpha_1,\ldots,\alpha_m;\beta_1,\ldots, \beta_m;z),$
\begin{equation*} 
\mathbb{E}\{T^m\}=\frac{m!\mathrm{e}^{\lambda/\nu}}{(\lambda+\mu)^m}
\cdot {_mF_m}(a,\ldots,a; 1+a,\ldots,1+a;-\lambda/\nu),
\end{equation*}
where $a=(\lambda+\mu)/\nu$.
In particular, the mean of $T$ is given by
\[
\mathbb{E}\{T\}=\frac{\mathrm{e}^{\lambda/\nu}}{\lambda+\mu}\Phi\left(\frac{\lambda+\mu}{\nu}, 1+\frac{\lambda+\mu}{\nu}; -\lambda/\nu\right),
\]
where $\Phi$ is a confluent hypergeometric function, 
see \cite{GR}, formula (9.210). 
By \cite{GR}, formula (8.354.1),
   $\mathbb{E}\{T\}$ could be written in the equivalent form:
\[
\mathbb{E}\{T\}=\mathrm{e}^{\lambda/\nu}
\sum_{n=0}^\infty\frac{(-1)^n(\lambda/\nu)^n}{n!(\lambda+\mu+n\nu)}
=\frac{\mathrm{e}^{\lambda/\nu}}{\nu}\left(\frac{\lambda}{\nu}\right)^{-\frac{\lambda+\mu}{\nu}}
\gamma\left(\frac{\lambda+\mu}{\nu}, \frac{\lambda}{\nu}\right),
\]
where $\gamma$ is the incomplete gamma-function.

 
\section{Piecewise linear process with two alternating patterns and a double jump component}
\label{sec:Cox/Hawkes  processes}
\setcounter{equation}{0}

Let $N_{m}=N_{m}(t),\;m\geq0,$ 
be  the sequence of independent Poisson processes
which are driven by two alternating sequences of parameters:
 $\vec\lambda^{(0)}=\{\lambda^{(0)}_n\}_{n\geq0}$ 
and  $\vec\lambda^{(1)}=\{\lambda^{(1)}_n\}_{n\geq0}.$ 
That is, $N_{m}(t)\equiv N(t; \vec\lambda^{(\varepsilon_m)}),$ see the definition in \eqref{def:Nm}.
Here $\varepsilon_m$ is a sequence of alternating $0$ and $1:$ $(0, 1, 0, 1, 0, \ldots)$ or
$( 1, 0, 1, 0, \ldots ).$

Let $\{T_m\}_{m\geq1}$ 
be the sequence of independent 
positive random variables, 
$T_m\geq0,$  $m\geq1,$ and 
$M=M(t)$ be the associated process \eqref{def:M}, that counts arrivals of
$T^{(+,m)}:=T_1+\ldots+T_m,\;T^{(+,0)}=0,$  till time $t,\;t>0.$

Let $\varepsilon=\varepsilon(t)\in\{0, 1\}$ be 
the  process, which indicates the current state as follows:
$\varepsilon(t)=\varepsilon_m$ for $T^{(+,m)}\leq t<T^{(+,m+1)},$ $m\geq0$.

We assume that   random variable   $T_m$ has
Poisson-modulated exponential distribution, 
\[T_m\sim\mathrm{PoExp}(\vec\lambda^{(\varepsilon_{m-1})}, \vec\mu^{(\varepsilon_{m-1})}),\]
based on the   Poisson process $N_{m-1}=N_{m-1}(t),\;m\geq1.$
Here 
$\vec\mu^{(0)}=\{\mu^{(0)}_n\}_{n\geq0}$ 
and 
$\vec\mu^{(1)}=\{\mu^{(1)}_n\}_{n\geq0}$ 
are the two sequences of switching intensities (see \eqref{eq:barF}-\eqref{def:xit}).
The alternating survivor functions of $T_m,\;m\geq1,$ due to \eqref{1}
are given by
\[\begin{aligned}
\overline{F^{(i)}}(t)=1-F^{(i)}(t)=\mathbb{P}\{T_m>t\}&=\mathbb{E}\left\{
\exp\left(
-\int_0^t\mu_{N(s)}^{(i)}\mathrm{d} s
\right)
\right\}\\
&=\sum_{n=0}^\infty\Lambda_n^{(i)}
a_n(t;\overrightarrow{\lambda^{(i)}+\mu^{(i)}}),
\qquad t\geq0,
\end{aligned}\]
 and the corresponding densities are defined by 
 $f^{(i)}(t)=-\mathrm{d} \overline{F^{(i)}}(t)/\mathrm{d} t,\quad
i=\varepsilon_{m-1}\in\{0, 1\}.$

In this section 
we study a piecewise linear process which follows
two patterns alternating after the holding times $T_m$. %
Precisely,  we define  the piecewise linear process $\mathbb L$
based on the two sequences of tendencies, $\{c^{(0)}(n)\}_{n\geq0}$ and $\{c^{(1)}(n)\}_{n\geq0}$,
 alternating at the time instants $T^{(+, m)},$ such that
 \begin{equation}
\label{def:T}
\mathbb L(t) 
=\sum_{m=1}^{M(t)}l_{m-1}(T_m)+l_{M(t)}\left(t-T^{(+, M(t))}\right),
\qquad t\geq0,
\end{equation}
 where
 \begin{equation}
\label{def:lm}
l_m(t)=\int_0^tc^{(\varepsilon_m)}(N_m(s))\mathrm{d}  s,\qquad m\geq0, t\geq0.
\end{equation}

This definition coincides with \eqref{def:L}.   A simulated sample path is presented by
Figure \ref{fig1}.

\begin{figure}[ht]
\begin{center}
  \includegraphics[scale=1.5]{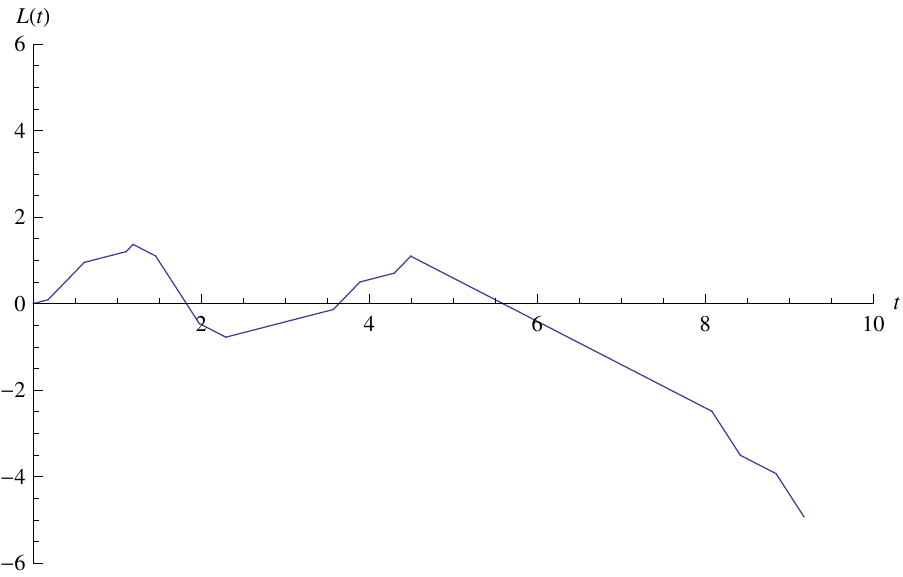}\\
 \caption{A sample path the piecewise linear process
 with two patterns of two pairs of alternating velocities. 
The 0-pattern with $c=0.5$ and 2.0; the 1-pattern with 
$c=-1.0$ and -3.0. In both cases, the inter-switching times 
are Poisson-modulated with $\lambda_n=1.5;\;\mu_n=1+n$
  }\label{fig1}
 \end{center}
\end{figure} 

The jump component added to this process consists of two parts.
The first one calculates the jumps occurring at the arrival times 
of the embedded Poisson processes. This corresponds to 
the case that each tendency switching (inside the time interval  $[T^{(+,m-1)},  \;T^{(+, m)})$) is 
accompanied with a jump of the magnitude $r_m(\cdot)$.
The 
compound Poisson processes
$j_m(t),\;t\geq0,$ 
\begin{equation*}
j_m(t)=\sum_{n=1}^{N_{m}(t)}r_m(n-1),\qquad m\geq0,
\end{equation*} 
presents the summed jump component.
Here $\{r_m(n)\},\;n\geq0, m\geq0,$ 
are independent random jump amplitudes, independent of counting process $N_m$.
Assume that the distributions of jumps are alternating,
such that the processes $j_m$ with even (odd)   $m$ 
are identically distributed.
Let
\begin{equation}
\label{def:j}
j(t)=\sum_{m=1}^{M(t)}j_{m-1}(T_m)+j_{M(t)}\left(t-T^{(+,M(t))}\right).
\end{equation}
counts the total number of this type of jumps.

 The second jump part is defined by  jumps occurring 
 at times $T^{(+,m)},\;m\geq1,$
 when   
 the pattern changes.
We assume that the jump amplitudes depend on the number of interventions 
$N_{m-1}(T_m),$
during the elapsed time $T_m,$
\begin{equation}
\label{def:J}
J(t)=\sum_{m=1}^{M(t)}R_{m-1}\left(N_{m-1}(T_m)\right).
\end{equation}
Here   independent  random  variables $\{R_m(n)\}$ 
are the jump magnitudes, which are independent of $T_m,$
$N_{m}$  and
$\{r_m(n)\},$ $n\geq0,$  $ m\geq0$.
Assume that  $R_m(\cdot)$ are of the alternating distributions.

Summarising, the jump component $\mathbb J(t)$
is defined by
\begin{equation}
\label{def:JJ}
\mathbb J(t)=j(t)+J(t)
=\sum_{m=1}^{M(t)}\left[
R_{m-1}\left(N_{m-1}(T_m)\right)+j_{m-1}(T_m)
\right]
+j_{M(t)}(t-T^{(+, M(t))}),\qquad t>0.
\end{equation}
Process $\mathbb J=\mathbb J(t)$ is an alternating renewal 
process, see e.g.   \cite{Cox-renewal}. The behaviour of the paths of such a process 
$X(t):=\mathbb L(t)+\mathbb J(t)$ is illustrated by Figure \ref{fig2}.

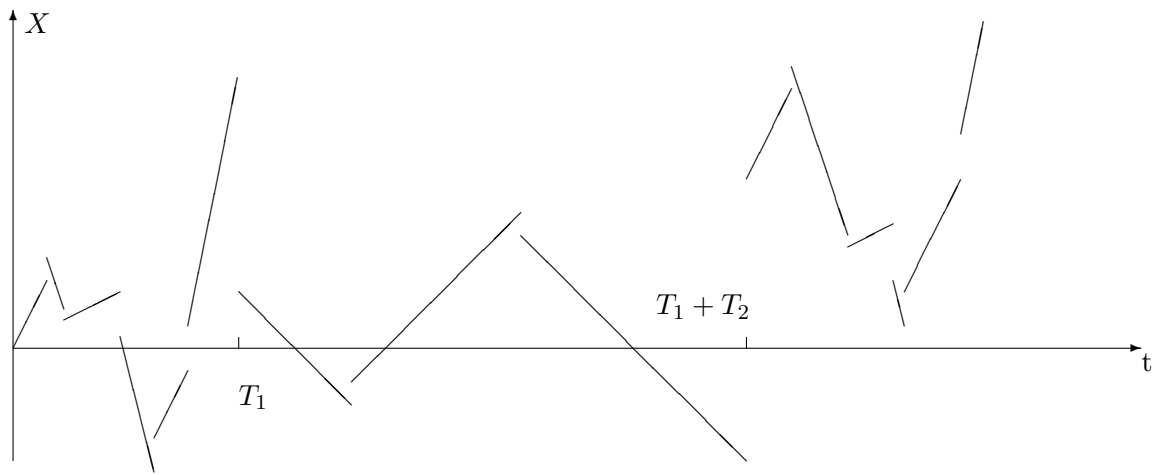
\begin{figure}[ht]
\setlength{\unitlength}{1.5cm}
\begin{picture}(10,4)(0,0)

\put(0,1){\vector(1,0){10}} \put(10,0.8){\makebox{t}}
\put(0.1,3.8){\makebox{$X$}} \put(0,0){\vector(0,1){4}}

\put(0,1){\line(1,2){0.3}}
\put(0.3,1.8){\line(1,-3){0.15}}
\put(0.45,1.25){\line(2,1){0.5}}
\put(0.95,1.1){\line(1,-4){0.3}}
\put(1.25,0.2){\line(1,2){0.3}}
\put(1.55,1.2){\line(1,5){0.44}}
\put(2,0.5){\makebox{$T_1$}}
\put(2,1){\line(0,1){0.1}}

\put(2,1.5){\line(1,-1){1}}

 \put(3,0.7){\line(1,1){1.5}}

\put(4.5,2){\line(1,-1){2}} 

\put(6.5,2.5){\line(1,2){0.4}} 
\put(6.9,3.5){\line(1,-3){0.5}} 
\put(7.4,1.9){\line(2,1){0.4}} 
\put(7.8,1.6){\line(1,-4){0.1}}

\put(7.9,1.5){\line(1,2){0.5}}
\put(8.4,2.9){\line(1,5){0.2}}

\put(6.5,1){\line(0,1){0.1}}
\put(5.7,1.3){\makebox{$T_1+T_2$}} 
\end{picture}
\caption{A sample path of the piecewise linear process with two alternating patterns and jumps}\label{fig2}
\end{figure}

In this paper we study the piecewise linear process $\mathbb L(t)$ accompanied
with the double  jump component 
$\mathbb J(t),$  defined by  \eqref{def:j}-\eqref{def:JJ}. Process
$X(t):=\mathbb L(t)+\mathbb J(t)$  successively passes through
the alternating states 
 determined by the two sets of parameters
\begin{equation}\label{def:sigma}
\sigma^{(i)}:=\langle
\vec c^{(i)},\;\vec r^{(i)},\; \vec R^{(i)},\;\vec\mu^{(i)},\;\vec\lambda^{(i)}\rangle,
\qquad i\in\{0, 1\}.
\end{equation}

Jump processes $j(t)$ and $J(t)$ are of different nature. Under the given trend $c=c^{(i)}(n)$
the jump $j(t)$ with the magnitude $r^{(i)}(n)$ always
occurs just after each tendency switching, while 
the jump with magnitude $R^{(i)}(n)$
occurs  in the case when the process changes the pattern,  \eqref{def:J}.

We analyse the expectation of $X(t)$.
Denote by
\[
\overline{R^{(i)} (n)}:=\mathbb{E}\{R^{(i)} (n)\},
\qquad \overline{r^{(i)} (n)}:=\mathbb{E}\{r^{(i)}(n)\},\quad m\geq0,\;n\geq0,\; i=\varepsilon_m,
\]
the expectations of jump values,
alternating with respect to $m,$ that is $\overline{R_m (\cdot)}$ and $\overline{r_m (\cdot)}$.

Denote   
\begin{equation}\label{def:rho}
\rho(n)=\rho_n^{(i)}:=\sum\limits_{k=0}^{n-1}\overline{r^{(i)}(k)},\quad n\geq 1,\qquad\rho(0)=0.
\end{equation}
Assume that for  both the states, $i\in\{0, 1\},$ 
the following series converge:     
\begin{equation}
\label{convergence:Deltak}
\begin{aligned}
\sum_{n=k}^\infty\rho(n)\Lambda_n\kappa_{n,k}(\vec\mu)<\infty,\\
\sum_{n=k}^\infty\left(c(n)+\mu_n\overline{R(n)}\right)\Lambda_n\kappa_{n,k}(\vec\mu)<\infty,
\end{aligned}
\qquad k\geq0.
\end{equation}  
This condition  extends condition \eqref{def:bk} 
fixing   relations between 
the sets of   ``observable'' parameters $c(n),\; \overline{R(n)},$  $\;\rho(n),\;n\geq0,$ 
and ``hidden'' intensity parameters
$\vec\lambda,\;\vec\mu$. 
Similarly to Eq. \eqref{def:bk}, conditions \eqref{convergence:Deltak} 
are sufficient for ``finite accumulation of jumps" at any finite time interval.

\begin{theo}
Let condition \eqref{convergence:Deltak} hold.

The \textup{(}conditional\textup{)} expectations 
$\mathfrak{M}_i
=\mathfrak{M}_i(t)
=\mathbb{E}\{\mathbb L(t)+\mathbb J(t)~|~\varepsilon(0)=i\},$
under the given initial set of parameters $\sigma$
solve the following coupled integral equations\textup{:}
\begin{equation}\label{eq:frakM}
\begin{aligned}
\mathfrak{M}_0(t)&
=\mathfrak m(t~|~\sigma^{(0)})
+\int_0^tf_T^{(0)}(u)\mathfrak{M}_1(t-u)\mathrm{d} u,\\
\mathfrak{M}_1(t)&
=\mathfrak m(t~|~\sigma^{(1)})
+\int_0^tf_T^{(1)}(u)\mathfrak{M}_0(t-u)\mathrm{d} u,
\end{aligned}\qquad t>0.\end{equation}
 Here $f_T^{(0)}$ and $f_T^{(1)}$ are the density functions of the 
 alternating distributions of $T_m,$ see \eqref{eq:fT1};
 function $\mathfrak m(t~|~\sigma)$
  is defined by the initial state $\sigma=\langle
c,\;\vec r,\;\vec R,\vec\mu,\;\vec\lambda
\rangle$ as follows\textup{:}
\begin{equation}
\label{def:m}
\mathfrak m(t~|~\sigma)
=\sum_{k=0}^\infty\frac{1-\mathrm{e}^{-(\lambda_k+\mu_k)t}}{\lambda_k+\mu_k}
\sum_{n=k}^\infty\Delta(n)\Lambda_n\kappa_{n,k}(\overrightarrow{\lambda+\mu}),
\qquad t\geq0,
\end{equation}
where $\Delta(n)=\Delta(n~|~\sigma)
=c(n)+\lambda_n\overline{r(n)}+\mu_n\overline{R(n)},\;n\geq0.$
\end{theo}

\begin{proof} 
Let  $\varepsilon_0=0$.

The following equality in law holds:
\[
\begin{aligned}
    \left[\mathbb L(t)+\mathbb J(t)\right]|_{\varepsilon(0)=0}& =\left[l_0(t)+j_0(t)\right]\mathbf 1_{\{T>t\}} 
    +\Big\{\left[l_0(T)+j_0(T)+R_0(N_0(T))\right] \\
    +&\left[ \mathbb L(t-T)+\mathbb J(t-T)\right]\Big\}|_{\varepsilon(0)=1}\mathbf 1_{\{T\leq t\}}. 
\end{aligned}
\]
Here $T=T_1$ is the time of the first pattern's switching.
 
 Fix the initial state  $\sigma=\langle
c,\;\vec{\overline{r}},\;\vec{\overline{ R}}, \vec\mu,\;\vec\lambda
\rangle.$ 
One can write
\begin{equation}\label{eq:K}
\begin{aligned}
\mathfrak{M}_0(t)  
=\mathbb E_0\left\{l_0(t)+j_0(t)\right\}\mathbf{1}_{\{T>t\}}
+&\int_0^t f_T^{(0)}(s)\mathbb{E}_0\left\{
l_0(s)+j_0(s)+R_0(N_0(s))\right\}\mathrm{d} s
\\
+&\int_0^tf_T^{(0)}(s)\mathfrak M_1(t-s)\mathrm{d} s.
\end{aligned}\end{equation}
(all parameters are of state $\varepsilon=0;$  
$\mathbb{E}_0$ is the conditional expectation with respect to the conditional probability
$\mathbb{P}_0\{\cdot\}:=\mathbb{P}\{\cdot~|~\varepsilon(0)=0\}$). 
The equation for $\mathfrak M_1(\cdot)$ is similar.

By using  \eqref{def:bk}, \eqref{1}-\eqref{2} 
due to \eqref{def:rho}
one can obtain
\begin{equation}\label{eq:Ej}
\mathbb{E}_0\left[j(t)\cdot\mathbf{1}_{\{T>t\}}\right]
=\sum_{n=1}^\infty\Lambda_n\rho(n)a_n(t, \overrightarrow{\lambda+\mu}),
\end{equation}
\begin{equation}\label{eq:ERj}
\int_0^t f_T^{(0)}(s)\mathbb{E}_0\left\{
j_0(s)+R_0(N_0(s))\right\}\mathrm{d} s
=\sum_{n=0}^\infty\left(\rho(n)+\overline{ R(n)}\right)\mu_n\Lambda_n
\int_0^ta_n(s; \overrightarrow{\lambda+\mu})\mathrm{d} s
\end{equation}

Further, 
\begin{equation}
\label{eq:dEL}\begin{aligned}
\frac{\mathrm{d}}{\mathrm{d} t}\mathbb{E}_0\left(
\mathbb L(t)\mathbf{1}_{T>t}
\right)
=&\sum_{n=0}^\infty\mathbb{E}_0\left(\frac{\mathrm{d} {\mathbb L}}{\mathrm{d} t}(t)\mathbf{1}_{N(t)=n}~|~T>t\right)
\mathbb{P}_0\{T>t\}
+\mathbb{E}_0\left(
\mathbb L(t)\cdot\frac{\mathrm{d}\mathbf{1}_{T>t}}{\mathrm{d} t}
\right)\\
=&\sum_{n=0}^\infty c^{(0)}(n)\mathbb{P}_0\{T>t,\;N(t)=n\}
-f_T^{(0)}(t)\mathbb{E}_0\left(
\mathbb L(t)\right).
\end{aligned}\end{equation}

Integrating \eqref{eq:dEL} by  \eqref{1} we get
\begin{equation*}
\begin{aligned}
\mathbb{E}_0\left(
\mathbb L(t)\mathbf{1}_{T>t}
\right)
=\sum_{n=0}^\infty c^{(0)}(n)\Lambda_n\int_0^ta_n(s; \overrightarrow{\lambda+\mu}) \mathrm{d} s
-\int_0^tf_T^{(0)}(s)\mathbb{E}_0\left(\mathbb L(s)\right)\mathrm{d} s.
\end{aligned}\end{equation*}
Therefore,
\begin{equation}\label{eq:ll}
\mathbb E_0\left\{l_0(t)\mathbf{1}_{\{T>t\}}\right\}
+\int_0^tf_T^{(0)}(s)\mathbb{E}_0\left(l_0(s)\right)\mathrm{d} s
=\sum_{n=0}^\infty c^{(0)}(n)\Lambda_n\int_0^ta_n(s; \overrightarrow{\lambda+\mu})\mathrm{d} s.
\end{equation}

Summing up \eqref{eq:Ej}, \eqref{eq:ERj} and \eqref{eq:ll} by using \eqref{eq:K}
we get equation \eqref{eq:frakM} with
\[
\mathfrak m(t~|~\sigma)
=\sum_{n=0}^\infty\rho(n)\Lambda_na_{n}(t; \overrightarrow{\lambda+\mu})
+\sum_{n=0}^\infty\Lambda_n\left(
(\overline{R(n)}+\rho(n))\mu_n+c(n)
\right)\int_0^ta_n(s; \overrightarrow{\lambda+\mu})\mathrm{d} s.
\]
for both the states $\sigma$.   

To complete the proof, we should convert the last expression into the form of  \eqref{def:m}.
Recalling \eqref{def:a},
function $\mathfrak m(t~|~\sigma)$  takes the form
$\mathfrak m(t~|~\sigma)=\mathfrak m_0+\mathfrak m_1(t),$
where $\mathfrak m_0$ is constant, 
\begin{equation*}
\mathfrak m_0=
\sum_{n=0}^\infty\Lambda_n\left(c(n)+\left(\overline{ R(n)}+\rho(n)\right)\mu_n\right)
\sum_{k=0}^n(\lambda_k+\mu_k)^{-1}\kappa_{n,k}(\overrightarrow{\lambda+\mu})
\end{equation*}
and
\begin{equation}\label{def:m1}
\begin{aligned}
\mathfrak m_1(t)&=\sum_{k=0}^\infty\mathrm{e}^{-(\lambda_k+\mu_k)t}\sum_{n=k}^\infty
\Lambda_n\kappa_{n,k}(\overrightarrow{\lambda+\mu})\left[
\rho(n)-(\lambda_k+\mu_k)^{-1}\left(c(n)+\left(\overline{ R(n)}+\rho(n)\right)\mu_n\right)
\right]\\
&=\sum_{k=0}^\infty\frac{\mathrm{e}^{-(\lambda_k+\mu_k)t}}{\lambda_k+\mu_k}
\sum_{n=k}^\infty\Lambda_n\kappa_{n,k}(\overrightarrow{\lambda+\mu})\left[
(\lambda_k+\mu_k)\rho(n)-c(n)-\left(\overline{R(n)}+\rho(n)\right)\mu_n
\right].
\end{aligned}\end{equation}

Then, notice that
\begin{equation}\label{def:m11}\begin{aligned}
\sum_{n=k}^\infty&\Lambda_n\kappa_{n,k}(\overrightarrow{\lambda+\mu})\left[
(\lambda_k+\mu_k)\rho(n)-\rho(n)\mu_n\right]\\
=\sum_{n=k}^\infty&\Lambda_n\kappa_{n,k}(\overrightarrow{\lambda+\mu})\rho(n)\left[
(\lambda_k+\mu_k)-(\lambda_n+\mu_n)+\lambda_n\right]\\
=-\sum_{n=k+1}^\infty&\Lambda_n\kappa_{n-1, k}(\overrightarrow{\lambda+\mu})\rho(n)
+\sum_{n=k}^\infty\Lambda_n\kappa_{n,k}(\overrightarrow{\lambda+\mu})\rho(n)\lambda_n\\
=-\sum_{n=k}^\infty&\Lambda_{n+1}\kappa_{n,k}(\overrightarrow{\lambda+\mu})\overline{r(n)}
=-\sum_{n=k}^\infty\Lambda_{n}\kappa_{n,k}(\overrightarrow{\lambda+\mu})\lambda_n\overline{r(n)}.
\end{aligned}\end{equation}
From \eqref{def:m1}-\eqref{def:m11} we obtain
\begin{equation*}
\mathfrak m_1(t)=-\sum_{k=0}^\infty\frac{\mathrm{e}^{-(\lambda_k+\mu_k)t}}{\lambda_k+\mu_k}
\sum_{n=k}^\infty\Lambda_n\kappa_{n,k}(\overrightarrow{\lambda+\mu})
\left(c(n)+\lambda_n\overline{r(n)}+\mu_n\overline{R(n)}\right).
\end{equation*}

Note that by definition $\mathfrak m(t~|~\sigma)|_{t=0}=0$. Therefore 
$
\mathfrak m_0=-\mathfrak m_1(0)
$
and 
\begin{equation*}
\mathfrak m(t~|~\sigma)=-\mathfrak m_1(0)+\mathfrak m_1(t)
=\sum_{k=0}^\infty\frac{1-\mathrm{e}^{-(\lambda_k+\mu_k)t}}{\lambda_k+\mu_k}
\sum_{n=k}^\infty\Delta(n)\Lambda_n\kappa_{n,k}(\overrightarrow{\lambda+\mu}).
\end{equation*}
\end{proof}

Martingale condition \eqref{eq:mart} follows from equations
 \eqref{eq:frakM}.

\begin{theo}\label{theo:martingale}
Let condition \eqref{convergence:Deltak}
 holds.
 
Process $X=X(t)$ is a martingale if and only if 
for both states the parameters of the model satisfy
\begin{equation}
\label{eq:martingale-condition}
\Delta(n):=c(n)+\lambda_n \overline{r(n)}+\mu_n\overline{R(n)}=0.\qquad n\geq0.
\end{equation}
\end{theo}

\begin{proof}
By renewal character of the process 
$X=X(t)$ is a martingale if and only if 
the expectations vanish, $\mathfrak M_i(t)\equiv0,\;i\in\{0, 1\},$
or, equivalently,
$\mathfrak m(t~|~\sigma)\equiv0$ for the both states, 
$\sigma\in\{\sigma^{(0)},\;\sigma^{(1)}\},$ see \eqref{eq:frakM}.
By \eqref{def:m} this is equivalent to
\begin{equation}
\label{eq:DD}
\sum_{n=k}^\infty\Lambda_n\Delta(n)
\kappa_{n,k}(\overrightarrow{\lambda+\mu})=0,\qquad
k=0,\;1,\;2,\ldots
\end{equation}

If  $\Delta(n)=0$ $\forall n,\;n\geq0,$ then \eqref{eq:DD} holds.

On the other hand, from \eqref{eq:DD} one can obtain
$\Delta(n)=0\;\forall n\geq0.$
Indeed, summing up these equations by using 
 Vandermonde properties \eqref{Vandermonde}
 we have
 \[
0= \sum_{k=0}^\infty\left[
 \sum_{n=k}^\infty\Lambda_n\Delta(n)
\kappa_{n,k}(\overrightarrow{\lambda+\mu})
 \right]
 =\sum_{n=0}^\infty\Delta(n)\Lambda_n
 \sum_{k=0}^n\kappa_{n,k}(\overrightarrow{\lambda+\mu})
 =\Delta(0)\kappa_{0,0}=\Delta(0).
 \]
Hence $\Delta(0)=0$. 

Then, we prove $\Delta(m)=0,\;m=1, 2, \ldots,$ for both the states by induction.
This follows 
by applying  \eqref{Vandermonde} to the sequential sums, $m=1, 2, \ldots,$
\[
0=\sum_{k=0}^\infty(\lambda_k+\mu_k)^m\left[
 \sum_{n=k}^\infty\Lambda_n\Delta(n)
\kappa_{n,k}(\overrightarrow{\lambda+\mu})
 \right]
 =\sum_{n=0}^\infty\Lambda_n\Delta(n)
 \sum_{k=0}^n(\lambda_k+\mu_k)^m\kappa_{n,k}(\overrightarrow{\lambda+\mu}).
\]
\end{proof}

\begin{rem}
For some state of the process, let 
 the supports of jump amplitudes $r(n)$ and $R(n)$ and the tendency $c(n)$
be situated in the same semi-line. 
In this case\textup{,} by Theorem \textup{\ref{theo:martingale},} the process $X$ is not a martingale.

Precisely\textup{,} the equivalent martingale measure for 
process $X$ does not exist
in the following   cases\textup{:}  in some state of the process\textup{,} $i\in\{0, 1\}, \;\exists n$ 
\begin{equation}
\label{support-}
c(n)<0, \qquad\text{and}\qquad\mathrm{supp}\{r(n)\},\;\mathrm{supp}\{R(n)\}\subset(-\infty,\;0);
\end{equation}
 
\begin{equation}
\label{support+}
c(n)>0 \qquad\text{and}\qquad \mathrm{supp}\{r(n)\},\;\mathrm{supp}\{R(n)\}\subset(0,\;+\infty).
\end{equation}

The problem of  existence of equivalent martingale measures 
is discussed in the next section.
\end{rem}

\begin{rem}
Theorem \textup{\ref{theo:martingale}} seems natural.
For instance\textup{,} if the holding times $T_m$ are independent of $N_m$ and
exponentially distributed 
 with alternating parameters $\mu^{(i)},$
condition \eqref{eq:martingale-condition} 
\textup{(}with $\mu^{(i)}$ instead of $\mu_n$\textup{)} 
characterises a martingale for the similar piecewise linear process with double jump component\textup{,}
see  \textup{\cite{ALEA17},  Corollary 3.1.}
Moreover\textup{,}  condition \eqref{eq:martingale-condition}   is very similar
to the martingale condition for the simple jump-telegraph model, \eqref{eq:mart-classic}\textup{,}
see  \textup{\cite{Quant,LR,KR}.}

In \eqref{eq:martingale-condition}
the term $\lambda_n \overline{r(n)}$  corresponds to the correction of tendency $c(n),$
which is provoked by jumps occurring at each tendency switching\textup{,} whereas
the term $\mu_n \overline{R(n)}$ corresponds to
the jumps accompanying the changes of patterns.

\end{rem}

\section{Market model}\label{sec:market}
\setcounter{equation}{0}

We consider the financial market  model based on
the  piecewise linear stochastic process with jumps
$X=X(t)=\mathbb L(t)+\mathbb J(t),$ $t\geq0,$ which is
 defined  on the filtered probability  space $(\Omega, \mathcal F, \{\mathcal F_t\}_{t\geq0}, \mathbb{P})$
 by \eqref{def:T}, \eqref{def:j} and \eqref{def:J}.
The alternating states of the process $X$ are described by the two sets of parameters
\begin{equation*}
\langle
\vec  c^{(i)},\; \vec r_m,\; \vec R_m,\; \vec\mu^{(i)},\;\vec\lambda^{(i)}
\rangle,\qquad i\in\{0, 1\},\qquad m\geq0.
\end{equation*}
Here $\vec c^{(i)}=\{c^{(i)}(n)\}_{n\geq0},$
$\vec \mu^{(i)}=\{\mu^{(i)}_n\}_{n\geq0}$ and $\vec \lambda^{(i)}=\{\lambda^{(i)}_n\}_{n\geq0}$ 
are deterministic, $\mu^{(i)}_n, \lambda^{(i)}_n>0,\;i\in\{0, 1\}, n\geq0,$ and 
independent  random 
jump amplitudes $\vec r_m=\{ r_m(n)\}_{n\geq0}$ and $\vec R_m=\{ R_m(n)\}_{n\geq0}$
are greater then $-1,$
\[\mathrm{supp}\{r_m (n)\}\subset(-1,\;\infty),\qquad
\mathrm{supp}\{R_m(n)\}\subset(-1,\;\infty),\quad m\geq1,\;n\geq0.\]

Consider a market model of two risky primary assets,
stock and bond. 

The bond price is based on the continuous 
piecewise linear process 
\begin{equation*}
\mathbb Y(t)=\int_0^ty^{(\varepsilon(u))}\mathrm{d} u,\qquad t\geq0,
\end{equation*}
where $y^{(0)}\geq0$ and $y^{(1)}\geq0$ are continuously compounding interest rates
depending on the current market state. 
The bond price dynamics is defined by
\begin{equation}
\label{def:bond}
B(t)=\exp(\mathbb Y(t)),\qquad t\geq0.
\end{equation}

The stock price is defined by the stochastic exponential of $X$.
Precisely, denote
\begin{equation}
\label{def:zm}
z_m(t):=\prod_{n=1}^{N_m(t)}\left(1+r_m(n)\right),\qquad m\geq0,
\end{equation}
the stochastic exponential of the independent
compound Poisson process $j_m(t):=\sum_{n=1}^{N_m(t)}r_m(n)$.
The stock price is defined by 
\begin{equation}\label{def:stock}
S(t)=S_0\mathrm{e}^{\mathbb L(t)}
\prod_{m=1}^{M(t)}\left[
\left(1+R_{m-1}\left(N_{m-1}(T_m)\right)\right)z_{m-1}(T_{m})\right]\cdot z_{M(t)}\left(t-T^{(+,M(t))}\right).
\end{equation} 

Model based on  \eqref{def:bond}, \eqref{def:stock}
is characterised by the multiple sources of uncertainty, the Poisson processes $M$ 
and $N_m,$ $ m\geq0,$
 which make the model incomplete.  
Moreover, the model 
has discontinuities of unpredictable type and size. This
 makes the model essentially incomplete,  see  \cite{bardhan2}.

This model 
generalises the jump-telegraph model studied by \cite{Quant}.
See  \cite{KR} for the detailed presentation.

On the filtered probability  space 
$(\Omega, \mathcal F, \{\mathcal F_t\}_{t\geq0}, \mathbb{P})$ 
we define the equivalent measure 
$\mathbb{P}_*$ by means of the following Girsanov transform.

Consider the numerical (nonrandom) sequences 
$R^{(i)}_*(n),$  $r^{(i)}_*(n),$
such that $R^{(i)}_*(n),\;r^{(i)}_*(n)>-1,$ $n\geq0,\;i\in\{0, 1\}.$   Let 
\begin{equation}
\label{eq:girsanov}
c^{(i)}_*(n)=-\lambda_n^{(i)}r^{(i)}_*(n)-\mu_n^{(i)}R^{(i)}_*(n),
\qquad n\geq0,\; i\in\{0, 1\}.
\end{equation}

Define the piecewise linear process $\mathbb L_*$ 
with tendencies $c_*^{(i)}(\cdot)$ (as in Section \ref{sec:Cox/Hawkes  processes})
based on process $M=M(t),$
counting the states' switching, \eqref{def:T}-\eqref{def:lm},
\begin{equation*}
\mathbb L_*(t)=\int_0^tc_{*}^{(\varepsilon(s))}(N_{M(s)}(s))\mathrm{d} s
=\sum_{m=1}^{M(t)}l_{*,m-1}(T_m)+l_{*, M(t)}\left(t-T^{(+, M(t))}\right).
\end{equation*}
Here
\[
l_{*,m}(t)=\int_0^tc_*^{(\varepsilon_m)}(N_m(s))\mathrm{d} s,\qquad m\geq0.
\]

The  jump process $\mathbb J_*$   is defined by   
\begin{equation*}
j_{*,m}(t)=\sum_{n=1}^{N_m(t)}r_{*}^{(\varepsilon_{m-1})}(n),\qquad
j_*(t)=\sum_{m=1}^{M(t)}j_{*,m-1}(T_m)+j_{*,M(t)}(t-T^{(+, M(t))}),
\end{equation*}
and
\begin{equation*}
 J_*(t)
=
\sum_{m=1}^{M(t)}R^{(\varepsilon_{m-1})}_*(N_{m-1}(T_m)).
\end{equation*}
See \eqref{def:j}-\eqref{def:J}.

All processes  are based on the embedded state process $\varepsilon$ and 
$\varepsilon_m=\varepsilon(T^{(+,m)}).$ Due to \eqref{eq:girsanov} 
by Theorem \ref{theo:martingale}     
the sum $\mathbb L_*(t)+\mathbb J_*(t),\;t\geq0,$ is a martingale.

Let $Z(t)=\mathcal E_t(\mathbb L_*+\mathbb J_*),\; t>0.$
On the filtered probability space 
$(\Omega,\; \mathcal F,\;\{\mathcal F_t\}_{t\geq0},\;\mathbb{P}),\;t\geq0,$
define the equivalent measure  $\mathbb{P}_*$
by the Radon-Nikodym derivative
\begin{equation}
\label{eq:RN}
\begin{aligned}
\frac{\mathrm{d}\mathbb{P}_*}{\mathrm{d}\mathbb{P}}|_{t}
=Z(t)=&\mathrm{e}^{\mathbb L_*(t)}
\prod\limits_{m=1}^{M(t)}\left[
\left(1+R^{(\varepsilon_{m-1})}_*(N_{m-1}(T_m))\right)z_{*,m-1}(T_m)\right]
\cdot z_{*,M(t)}\left(t-T^{(+,M(t))}\right),
\quad t>0,
\end{aligned}\end{equation}
where $z_{*,m}$ are defined as \eqref{def:zm} (with $r_*^{(\varepsilon_m)}(n)$
instead of $r_m(n)$).

This means that for any $A\in\mathcal F_t$, 
where $\mathcal F_t$ is the natural filtration determined by $X$,
 we have 
 \begin{equation}
\label{def:Q}
\mathbb{P}_*(A)=\mathbb{E}\left[
Z(t)\mathbf{1}_{A}
\right],\quad t>0.
\end{equation}
It is easy to see that under measure $\mathbb{P}_*$ defined by
\eqref{eq:RN}, 
\emph{only} (unobservable) \emph{intensity parameters} $\vec\lambda^{(i)}$ and $\vec\mu^{(i)}$
of process $X$ are changed. In this circumstances \eqref{eq:RN}-\eqref{def:Q}
 may be treated as the \emph{Esscher transform}  $\mathbb{P}_*\sim\mathbb{P}$ (see  \cite{Elliott}).

The following result serves as a version of Cameron-Martin-Girsanov 
Theorem for this measure transformation.

\begin{theo}\label{theo:Esscher}
Under  measure $\mathbb{P}_*,$ which is defined by \eqref{eq:RN}-\eqref{def:Q}\textup{,}
the underlying state process $\varepsilon$
is governed by independent ${PoExp}(\vec\lambda_*, \vec\mu_*)$-distributed 
inter-switching times\textup{,}
with the alternating parameters 
\begin{equation}
\label{eq:la*}
\lambda_{*n}^{(i)}=\lambda^{(i)}_n(1+r^{(i)}_*(n))
\end{equation}
 and 
\begin{equation}
\label{eq:mu*}
\mu^{(i)}_{*n}
=\mu_n^{(i)}(1+R_*^{(i)}(n)),\;i\in\{0, 1\}.
\end{equation}
\end{theo}

Notice that under   measure $\mathbb{P}_*$ the unobservable 
intensity parameters $\vec\lambda^{(i)}$ and $\vec\mu^{(i)}$ of $X$
are transformed 
in agreement with the traditional results, related to simple jump-telegraph process,
 \cite{KR}, see also \cite{bardhan1}.

\begin{proof} 
The distribution of the first switching time  $T$ under measure $\mathbb{P}_*$ 
can be determined by the survivor functions 
$\mathbb{P}_*\{T>t,\; N(t)=n~|~\varepsilon(0)=i\},\;n\geq0, i\in\{0, 1\}$.
By definition \eqref{eq:RN}-\eqref{def:Q} 
\begin{equation*}
\begin{aligned}
\mathbb{P}_*\{T>t,\; &N(t)=n~|~\varepsilon(0)=i\}\\
=&\mathbb{E}\left(
\mathrm{e}^{l_*^{(i)}(t)}\prod_{k=1}^{N(t)}\left(1+r_*^{(i)}(k)\right)
\times\mathrm{e}^{-\xi^{(i)}(t)}\mathbf{1}_{\{N(t)=n\}}~|~\varepsilon(0)=i
\right)\\
=&\prod_{k=1}^{n}\left(1+r_*^{(i)}(k)\right)
\mathbb{E}\left(
\exp\left(
-\int_0^t\left(\mu^{(i)}_{N(s)}-c_*^{(i)}(N(s))\right)\mathrm{d} s
\right)\cdot\mathbf{1}_{\{N(t)=n\}}
\right).
\end{aligned}\end{equation*}
By \eqref{eq:psin1} and \eqref{eq:T>t-n}
 it follows that 
\begin{equation}\label{eq:FT*2}
\mathbb{P}_*\{T>t,\; N(t)=n~|~\varepsilon(0)=i\}
=\prod_{k=1}^{n}\left(1+r_*^{(i)}(k)\right)
\Lambda_na_n(t; \overrightarrow{\lambda^{(i)}+\mu^{(i)}-c_*^{(i)}}).
\end{equation}
Here $\overrightarrow{\lambda^{(i)}+\mu^{(i)}-c_*^{(i)}}:=\{\lambda_n^{(i)}+\mu_n^{(i)}-c_*^{(i)}(n)\}_{n\geq0}.$
Due to \eqref{eq:girsanov} 
\[
\lambda_n^{(i)}+\mu_n^{(i)}-c_*^{(i)}(n)
=\lambda_n^{(i)}\left(1+r_*^{(i)}(n)\right)+\mu_n^{(i)}\left(1+R_*^{(i)}(n)\right).
\]
Now, comparing \eqref{eq:FT*2} with \eqref{1}  we found that under measure $\mathbb{P}^*$
the first switching time $T$ has Poisson-modulated exponential distribution,
$T\sim\mathrm{PoExp}(\vec\lambda_*, \vec\mu_*),$  with parameters $\vec\lambda_*, \vec\mu_*$, which
are defined by   \eqref{eq:la*} and \eqref{eq:mu*}.

The  joint distribution of the  switching times, $T_1,\; T^{(+,2)}, \ldots, T^{(+,m)},$ 
could be represented similarly.
For instance, the joint distribution of $T_1$ and $T_1+T_2$ under measure $\mathbb{P}$ 
by independence of $T_1$ and $ T_2$ is given by 
\[
\mathbb{P}\left\{T_1\in\mathrm{d} s,\;T_1+T_2>t~|~\varepsilon(0)=i\right\}
=\mathbb{P}\left\{T_1\in\mathrm{d} s~|~\varepsilon(0)=i\right\}\cdot\mathbb{P}\left\{T_2>t-s~|~\varepsilon(0)=1-i \right\}.
\]
By Proposition \ref{prop-densities} one can obtain
\begin{equation}
\label{eq:P1}
\begin{aligned}
&\mathbb{P}\left\{T_1\in\mathrm{d} s,\;T_1+T_2>t~|~\varepsilon(0)=i\right\}
\\
=&\sum_{n=0}^\infty
\mu_n^{(i)}\Lambda_n^{(i)}a_n(s; \overrightarrow{\lambda^{(i)}+\mu^{(i)}})     
\cdot
\sum_{n=0}^\infty\Lambda_n^{(1-i)}a_n(t-s; \overrightarrow{\lambda^{(1-i)}+\mu^{(1-i)}})\mathrm{d} s,
\end{aligned}\end{equation}
see \eqref{1}-\eqref{2}.

On the other hand, under measure 
$\mathbb{P}_*$
 by definition \eqref{eq:RN}-\eqref{def:Q} we have for $s<t$
\begin{equation}
\label{eq:Q1}\begin{aligned}
&\mathbb{P}_*\left(
T_1\in\mathrm{d} s,\;T_1+T_2>t~|~\varepsilon(0)=i
\right)
=\mathbb{E}\left[
Z(t)\mathbf{1}_{\{T_1\in\mathrm{d} s,\;T_1+T_2>t\}}~|~\varepsilon(0)=i
\right]\\
=\mathbb{E}\Big[&\mathrm{e}^{l_*^{(i)}(s)+l_*^{(1-i)}(t-s)}\left(
1+R_*^{(i)}(N_0(s))\right)\\
&\times\prod_{k=0}^{N_0(s)}(1+r_*^{(i)}(k))\prod_{k=0}^{N_1(t-s)}(1+r_*^{(1-i)}(k))
\mathbf{1}_{\{T_1\in\mathrm{d} s,\;T_1+T_2>t\}}~|~\varepsilon(0)=i\Big].
\end{aligned}\end{equation}
Here
$l_*^{(i)}(s):=\int_0^sc_*^{(i)}(N_0^{(i)}(u))\mathrm{d} u$ and 
$l_*^{(1-i)}(t-s):=\int_0^{t-s}c_*^{(1-i)}(N_1(u))\mathrm{d} u$ are independent.

By applying again Proposition \ref{prop-densities}  we
confirm \eqref{eq:la*}-\eqref{eq:mu*}.
\end{proof}

By the fundamental theorem of market modelling
the model is arbitrage-free when 
the discounted stock price $\widetilde S(t):=B(t)^{-1}S(t)$
is a martingale under suitable equivalent measure $\mathbb{P}_*$ 
(see   \cite{HK,HP1,HP2} and  \cite{DS}).

Note that the discounted stock price $\widetilde S(t)$
is of the same form as $S(t)$ (with the alternating trends 
$\tilde c^{(i)}=c^{(i)}-y^{(i)},\;i\in\{0, 1\}).$
Hence, without loss of generality one can assume that the \emph{ interest rates are zeros},
$y^{(i)}=0,\;i\in\{0, 1\}.$ 

Meanwhile, Theorem \ref{theo:Esscher} permits to change the intensity parameters arbitrarily. 
Hence one can reach martingale condition \eqref{eq:martingale-condition}
by applying the Esscher transform \eqref{eq:RN}
only, if  values
$c(n),$ 
$\overline{r(n)}$ and $\overline{R(n)}$ are not of the same sign
(for all states and for all $n$).

If conditions \eqref{support-} or \eqref{support+}  hold, then
 model \eqref{def:stock}  has arbitrage opportunities. 

\begin{rem}
Market model  \eqref{def:stock} might be interpreted as follows. 

Assume that the interest rates are zeros.
Let 
\begin{equation}
\label{def:S-ess}
S_0\mathcal E_t(\mathbb L+J)
=S_0\mathrm{e}^{\mathbb L(t)}
\prod_{m=1}^{M(t)}
\left(1+R_{m-1}\left(N_{m-1}(T_m)\right)\right),\qquad t\geq0,
\end{equation}
be the essential component of asset price \eqref{def:stock}\textup{,} 
which is determined by inherent market forces.
This component takes into account only the jumps accompanying the pattern's switchings.

We assume that in the both states the tendency $c(n)$ and the accompanying 
jump amplitude $\overline{R(n)}$ are of the opposite signs\textup{:}
\begin{equation}\label{eq<0}
c(n)/\overline{R(n)}<0,\qquad\forall n.
\end{equation}
This assumption seems natural, since the market model with the stock price defined by
\eqref{def:S-ess} is arbitrage-free if and only if this condition holds, 
see  \textup{\cite{Quant,KR}.}

The jumps defined by $z_m,$ \eqref{def:zm}\textup{,} 
accompanying each tendency fluctuation inside the current pattern\textup{,}
can be considered  as the result of external interventions of small markets players. 
The stock price follows 
\begin{equation*}
S(t)=S_0\mathcal E_t(\mathbb L+J)\times \left[\prod_{m=1}^{M(t)}z_{m-1}(T_m)\cdot 
z_{M(t)}\left(t-T^{(+, M(t))}\right)\right],
\qquad t\geq0.
\end{equation*}

If condition \eqref{eq<0} is fulfilled, then 
the model is still arbitrage-free\textup{,} in spite of the efforts of small players.
The equivalent martingale measure can be provided 
by the Esscher transform \eqref{eq:RN}
with parameters $c_*(n),\;R_*(n)$ and $r_*(n)$ satisfying 
\eqref{eq:girsanov} in both the states. 
To define this transform we first assume that
under the martingale measure the modulation intensity $\lambda_{*n}$
can be determined 
 neglecting the  influence of external interventions\textup{,}
such that 
$c(n)+\lambda_{*n}\overline{r(n)}$ is of the same sign 
with $c(n)$\textup{:}
\[
\lambda_{*n}<\left|c(n)/\overline{r(n)}\right|.
\]
By \eqref{eq:la*} $\lambda_{*n}=\lambda_n(1+r_*(n))$. Hence\textup{,} the parameters $r_*(n)$
satisfy the condition
\[
-1<r_*(n)<-1+\lambda_n^{-1}\left|c(n)/\overline{r(n)}\right|.
\]

Then\textup{,} the Radon-Nikodym derivative \eqref{eq:RN}  provides the martingale measure $\mathbb{P}_*,$ if
\textup{(}see \eqref{eq:martingale-condition}\textup{)}
\[
c(n)+\lambda_{*n}\overline{r(n)}+\mu_{*n}\overline{R(n)}=0,
\]
which by \eqref{eq:la*}-\eqref{eq:mu*}
gives
\[
R_*(n)=-1-\frac{c(n)+\lambda_n(1+r_*(n))\overline{r(n)}}{\mu_n\overline{R(n)}},\qquad R_*(n)>-1.
\]
\end{rem}

\begin{ex}
Let us consider the model with deterministic jump values $r^{(i)}(n)$ and $R^{(i)}(n)$.
The model is arbitrage-free\textup{,} if 
for each $n,\;n=0, 1, 2, \ldots$and $i,\;i\in\{0, 1\},$ 
the triplets 
\[
c^{(i)}(n),\quad r^{(i)}(n)\quad \text{and} \quad R^{(i)}(n)
\]
are not of the same sign.
In this case  martingale measures $\mathbb{P}_*$ are defined by 
 transform  \eqref{eq:RN} with parameters $c_*(n),\; r_*(n)$ and $R_*(n)$ satisfying \eqref{eq:girsanov}  and in  both the states
\[
c(n)+\lambda_n(1+r_*(n))r(n)+\mu_n(1+R_*(n))R(n)\equiv0.
\]

Let the Poisson modulation in the model be unobservable. This can occur in the case 
of   constant trends and jumps $R$\textup{:} in  both the states
\[
c(n)\equiv c,\qquad R(n)\equiv R,
\]
with modulated hidden parameters  $\lambda(n)$ and $\mu(n)$. We assume also that
the parameters' modulation is not accompanied with jumps\textup{:} $r(n)\equiv0.$

In this case the model is  arbitrage-free\textup{,} if $c/R<0.$ 
The   martingale measures are defined by \eqref{eq:girsanov} 
with arbitrary $r_*(n)>-1$ and unique $R_*(n),$
defined by
\[
R_*(n)=-1-\frac{c}{R\mu_n},\qquad n\geq0.
\]
\end{ex}

\section{Concluding remarks}

In this paper, we 
introduce a rather new class of double stochastic piecewise linear processes 
with a Poisson modulated exponential distributions of persistent epochs. 

The first layer of stochasticity is a usual telegraph process 
based on an alternating Poisson process $N(t)$, 
and the second one (driving the change of patterns) 
is characterised by exponentially distributed holding times with a $N(t)$-modulated parameter.

This class of processes is exploited for the purposes of financial modelling.
In particular, we study 
an incomplete financial market model 
based on a jump-telegraph process. 
The dynamics of the considered stochastic process is characterised 
by two alternating types of tendencies, 
whose holding times have Poisson-modulated exponential distribution. 
Moreover, the model includes two different kinds of jumps, 
one occurring at the tendency switchings 
and the other at the changes of patterns. 
A relevant aspect of this model is the presence of external shocks, 
which affect the rates of the trend switchings. 
This feature ensures that the model is largely flexible 
and thus it is suitable to describe a wide family of financial market scenarios. 
Specifically, it can be used to describe a financial market model 
whose log-returns are influenced both by market forces and by efforts of speculators. 
The main results provided for the piecewise linear process with jumps 
include the determination of integral equations 
for the conditional means and the martingale conditions. 
The applications to the financial market model have been presented in detail, 
including conditions leading to an arbitrage-free model.

\subsection*{\bf Acknowledgements}
The authors thank two anonymous referees for their 
useful comments that improved the paper.
This research is partially supported by the group GNCS of INdAM, and by   
MIUR (PRIN 2017, project ``Stochastic Models for Complex Systems''). 

There are no conflicts of interest to this work.

\end{document}